\documentclass[12pt]{article}

\usepackage[utf8]{inputenc}
\usepackage[english]{babel}

\usepackage{lmodern}
\usepackage[T1]{fontenc}

\usepackage{microtype}

\usepackage[top=1.5cm,bottom=1.5cm,left=1.5cm,right=1.5cm]{geometry}

\usepackage{amsthm,amssymb,amsmath}
\usepackage{bbm}

\usepackage{mathtools} 

\usepackage[pdfencoding=unicode, psdextra, citecolor=blue, urlcolor=blue,
  linkcolor=blue, colorlinks=true, bookmarksopen=true]{hyperref}
  
\usepackage{graphicx}
\graphicspath{{pics/}}
\usepackage{caption}
\usepackage{subcaption}

\usepackage{csquotes}
\usepackage[backend=biber,sorting=none]{biblatex}
\DeclareMathOperator{\Int}{int}
\DeclareMathOperator{\sgn}{sgn}

\newcommand{\R}{\mathbb{R}}
\newcommand{\Z}{\mathbb{Z}}
\newcommand{\polar}{\circ}
\newcommand{\const}{\mathrm{const}}
\newcommand{\id}{\mathbbm{1}}
\newcommand{\argmax}{\operatornamewithlimits{arg\,max}}
\newcommand{\Hb}[1]{\mathbb{H}_{#1}}

\newtheorem{proposition}{Proposition}
\newtheorem{theorem}{Theorem}
\newtheorem{corollary}{Corollary}
\newtheorem{remark}{Remark}
\newtheorem{definition}{Definition}

\newtheorem*{massumption}{Main Assumption}

\title{Explicit formulae for geodesics in left--invariant sub--Finsler problems on Heisenberg groups via convex trigonometry.}
\date{}
\author{L.V. Lokutsievskiy}

\begin{document}

\maketitle

\begin{abstract}
	In the present paper, we obtain explicit formulae for geodesics in some left--invariant sub--Finsler problems on Heisenberg groups $\Hb{2n+1}$. Our main assumption is the following: the compact convex set of unit velocities at identity admits a generalization of spherical coordinates. This includes convex hulls and sums of coordinate 2--dimensional sets, all left--invariant sub--Riemannian structures on $\Hb{2n+1}$, and unit balls in $L_p$--metric for $1\le p\le\infty$. In the last case, extremals are obtained in terms of incomplete Euler integral of the first kind.
\end{abstract}

\section*{Introduction}

The left--invariant sub--Finsler problem on the Heisenberg group $\Hb{3}$ of smallest dimension $3$ was studied for the first time by Herbert Busemann in~\cite{Busemann} (1947). Busemann was studying Dido's problem on the plane~$\R^2$ equipped by an arbitrary Finsler metric. All closed geodesics on $\Hb{3}$ were found in this paper, and this allowed Busemann to find the exact isoperimetric inequalities on Finsler planes. It is interesting that this problem was solved by Brunn-Minkowski theory and without using the Pontryagin maximum principle (or PMP for short), which was not yet discovered. It is worth to mention that full description of sub--Finsler geodesics on $\Hb{3}$ is not given in~\cite{Busemann}, since the author have been interested in isoperimetric inequalities and does not formulate the problem in term of sub--Finsler geometry.

The first full description of sub--Finsler geodesics on $\Hb{3}$ was obtained by Valerii N. Berestovskii with the help of Pontryagin's maximum principle in~\cite{BerestovskiiHeisenberg} (1994). Among other thing, he has found non-closed geodesics, which are surprisingly not necessarily straight lines if a compact convex set of unit velocities is not strictly convex. The work~\cite{BerestovskiiHeisenberg} naturally continued the work~\cite{VershikGershkovich}, where, for example, sub--Riemannian wave geodesic front on $\Hb{3}$ was constructed. sub--Riemannian geodesics on $\Hb{3}$ was also studied in~\cite{BrocketDai}. Also, left--invariant sub--Riemannian problems on Heisenberg groups $\Hb{2n+1}$ of higher dimensions were studies in~\cite{Gaveau1977}.

The definitions of sub--Riemannian and sub--Finsler geometries are very close to the definition of Riemannian geometry. The main differences are the following. In Riemannian geometry on a manifold~$M$, we assume that in any tangent space $T_qM$, there is given a centered at the origin ellipsoid $U(q)\subset T_qM$ of full dimension that represents the set of unit velocities at $q$. The assumption of $U(q)$ smooth dependence on $q$ allows us to compute length of curves, and the result is $M$ becoming a metric space with a series of very well-known properties. In sub--Riemannian geometry, we allow the ellipsoid $U(q)$ to have dimension that is smaller than $\dim M$ (it is still needs to be centered at the origin). In other words, $U(q)$ is an ellipsoid lying in a subspace $\Delta(q)\subset T_q M$, which depends smoothly on $q$ (for the exact definition, we refer the reader to the great modern book~\cite{AgrachevBarilariBoscain}). Using $U(q)$ we are not able to measure lengths of arbitrary curves on $M$ (for exmample, if $\dim\Delta(q) < \dim M$), but we are able to measure lengths of so called horizontal curves: a Lipschitz continuous curve on $M$ is called horizontal if it is a.e.\ tangent to~$\Delta(q)$. Hence if distribution $\Delta$ is nonholonomic (i.e.\ $(\mathrm{Lie}\,\Delta)(q)=T_qM$ for all $q\in M$), then any two points on $M$ can be connected by a horizontal curve. Again, using $U(q)$ as a set of unit velocities at $q$, we are able to introduce sub--Riemannian distance on $M$ --- infimum of lengths of horizontal curves joining the given two points. So, $M$ again becomes a geodesic metric space, but properties of sub--Riemannian geometry differ a lot from properties of classical Riemannian geometry. For example, Hausdorff dimension of a sub--Riemannian manifold usually differs from the topological one.

In sub--Finsler geometry sets $U(q)$ are not necessary ellipsoids, but instead can be arbitrary compact convex sets lying in $\Delta(q)$ and containing the origin in their (relative) interior. In recent years, interest to sub--Finsler geometry has greatly increased. This is related to the famous Gromov theorem on groups with polynomial growth~\cite{GromovPolynomialGrowth}, Berestovskii result on intrinsic left--invariant metrics on Lie groups (see \cite{Berestovskii1988}), and some other results including the latest results on hyperbolic geometry~\cite{FreemanLeDonne}. For example, recently it has been proved that sub--Finsler Carnot groups are the only locally compact, geodesic, isometrically homogeneous, and self-similar metric spaces (we refer the reader to~\cite{LeDonneFinsler} for details).

In the present paper, we work only with sub--Finsler structures on Heisenberg groups $\Hb{2n+1}$ (or $\mathbb{H}^n_{\mathbb{C}}$ in \cite{FreemanLeDonne} notations). Precisely, we are interested in obtaining explicit formulae for geodesics. As it was mentioned, full description of sub--Finsler geodesics on $\Hb{3}$ was obtained in~\cite{BerestovskiiHeisenberg}. For Heisenberg groups $\Hb{2n+1}$ of higher dimensions, only sub--Riemannian geodesics are known. There are some interesting results on $\Hb{2n+1}$ sub-Finser metric (e.g.\ Balogh and Calogero have proved that the only infinite geodesics on $\Hb{2n+1}$ for general strictly convex sub-Finsler metrics are straight lines \cite[Theorem 1.2]{BaloghCalogero}), but explicit formulae for geodesics are unknown. In the present paper, we use a newly developed machinery of convex trigonometry, which has been also used to obtain explicit formulae for geodesics in 5 sub--Finsler problems (including left-invriant problems on Engel and Cartan nilpotent Lie groups) for arbitrary two-dimensional compact convex sets $U$ (see \cite{CT1}). Moreover, recently, this machinery allows us (together with Yu.L.~Sachkov and A.A.~Ardentov) to obtain such formulae for all left--invariant sub--Finsler problems on $SL(2)$, $SU(2)$, $SE(2)$, and $SH(2)$ (see~\cite{CT2}). In the last paper, explicit formulae are also obtained for Finsler geodesics on the Lobachevsky plane, for the ball rolling problem on a Fisler plane, and for a series of yachts problems.

We are able to obtain the mentioned results in all these problems, since the set of unit velocities is a compact convex 2--dimensional set in all these problems, and machinery of convex trigonometry allows to work with these sets very conveniently. In the present paper, we are able to obtain explicit formulae for geodesics on $\Hb{2n+1}$ in terms of convex trigonometry when the $2n$--dimensional set of unit velocities admit a generalization of spherical coordinates. 

A brief introduction to convex trigonometry is given in Sec.~\ref{sec:ct}. The main result is Theorem~\ref{thm:main}, which is given in Sec.~\ref{sec:explicit_formulae}. Applications are given in Sec.~\ref{sec:lp}, \ref{sec:convex_hull_and_direct_product}, and~\ref{sec:subriemannain}.

\section{Main assumption on left--invariant sub--Finsler problems on Heisenberg groups}

The Heisenberg group $\Hb{2n+1}$ is defined as follows: $\Hb{2n+1}=\{q=(x,y,z):x\in\R^n,y\in\R^{n*},z\in\R\}$, and the group structure is given by the multiplication
\[
	q^1\cdot q^2 = (x^1,y^1,z^1)\cdot(x^2,y^2,z^2) = \left(x^1+x^2,y^1+y^2,z^1+z^2 + \frac12(\langle x^1,y^2\rangle - \langle x^2,y^1\rangle)\right).
\]
Identity $\id\in \Hb{2n+1}$ is $(0,0,0)$. Group $\Hb{2n+1}$ is a matrix Lie group:
\[
	\Hb{2n+1} = \left\{
			q=\left(\begin{array}{cccc}
				1 & x & \tilde z\\
				0 & \id_n & y\\
				0 & 0 & 1
			\end{array}\right)\in \mathrm{SL}(n+2,\R)
		\right\}
\]
where $\tilde z=z+\frac12\langle x,y\rangle$ and $\id_n$ denotes $n\times n$ identity matrix.

The classical left--invariant distribution $\Delta$ on $\Hb{2n+1}$, $\Delta(q)\subset T_q\Hb{2n+1}$, is given by the canonical 1-form $\alpha = dz - \frac12\sum_{i=1}^n (x_i\,dy_i-y_i\,dx_i)$ where $x=(x_1,\ldots,x_n)$ and $y=(y_1,\ldots,y_n)$, i.e\ 
\[
	\Delta(q)=\left\{(\dot x,\dot y,\dot z)\in T_q\Hb{2n+1}:\dot z=\frac12(\langle x,\dot y\rangle - \langle \dot x,y\rangle)\right\}=q\Delta(\id).
\] 

Any corresponding to $\Delta$ left--invariant sub--Finsler problem on $\Hb{2n+1}$ is given by a compact convex set $U\subset\Delta(\id)$ containing $0$ in its interior, $0\in\Int U$:
\begin{equation}
\label{eq:main_problem}
	\begin{array}{c}
		T\to\min;\\
		\dot q(t) \in q(t)U;\\
		q(0)=\id;\qquad q(T)=q_1.
	\end{array}
\end{equation}
Obviously, $(\mathrm{Lie}\,\Delta)(q)=T_q\Hb{2n+1}$, since $\alpha\wedge (d\alpha)^n$ does not vanish. Therefore, since $0\in\Int U$, for any $q_1\in \Hb{2n+1}$, there exists a trajectory connecting $\id$ and $q_1$ by the Rashevski-Chow theorem (see~\cite[Theorem~5.2]{AgrachevSachkov}), and, hence, there exists an optimal solution to problem~\eqref{eq:main_problem} by the Filippov theorem (see~\cite[Theorems~1 and~3 in Section~2.7]{Filippov}). This solution is not unique in general, but any solution to~\eqref{eq:main_problem} must obey Pontryagin maximum principle (see~\cite[Theorem 12.1]{AgrachevSachkov}).

Our purpose is to obtain explicit formulae for solutions to PMP. We are able to do this when the set of admissible controls~$U$ admits a generalization of spherical coordinates. Precisely, if $U$ satisfies the following

\begin{massumption}
	There exist compact convex sets $\Omega_i\subset\R^2$ with $0\in\Int\Omega_i$, $i=1,\ldots,n$, and a continuous convex positively homogeneous function $\mu:\R^n_+\to\R_+$ that is monotonously decreasing in each argument and strictly positive outside the origin such that
	\begin{equation}
	\label{eq:U_spherical}
		U = \{(\dot x,\dot y): \mu(\mu_{\Omega_1}(\dot x_1,\dot y_1),\ldots \mu_{\Omega_n}(\dot x_n,\dot y_n))\le 1\},
	\end{equation}
	where $\mu_\Omega$ denotes the Minkowski functional of a set $\Omega$.
\end{massumption}

Now, we try to explain why the main assumption is considered as an analogue of spherical coordinates on $\partial U$. For example, if we take the sub--Riemannian case for $n=2$, in which $U=\{\dot x_1^2+\dot y_1^2+\dot x_2^2+\dot y_2^2\le 1\}$, then $U$ satisfies the main assumption with $\Omega_i$ being the unit discs and $\mu(\lambda_1,\lambda_2)=(\lambda_1^2+\lambda_2^2)^{1/2}$. The optimal control always belongs to the sphere $\partial U$, which can be described by trigonometric functions $(\cos\theta_i,\sin\theta_i)\in\partial\Omega_i$: since $\lambda_1^2+\lambda_2^2=1$ on $\partial U$, we put $\lambda_1=\cos\zeta$ and $\lambda_2=\sin\zeta$ and obtain the following classical spherical coordinates on $\partial U$:
\[
	\dot x_1 = \cos\zeta \cos\theta_1;\quad \dot y_1=\cos\zeta\sin\theta_1;\quad
	\dot x_2 = \sin\zeta \cos\theta_2;\quad \dot y_2=\sin\zeta\sin\theta_2.
\]

We are able to repeat this procedure by functions of convex trigonometry for arbitrary $U$, if it satisfies the main assumption.

\begin{proposition}
	If a set $U$ satisfies the main assumption, then it is a compact convex set containing the origin in its interior.
\end{proposition}

\begin{proof}
	Indeed, $0\in\Int\Omega$, since $\mu(\mu_{\Omega_1}(0,0),\ldots, \mu_{\Omega_n}(0,0))=0<1$ and functions $\mu$, $\mu_{\Omega_i}$ are continuous (see~\cite[Theorem 10.1]{Rockafellar}). Obviously, set $U$ is closed. Set $U$ is compact, since $\mu$ is positively homogeneous: there exists $c>0$ such that $\mu(\lambda_1,\ldots,\lambda_n)>1$ when $\sum_i\lambda_i\ge c$, $\lambda_i\ge 0$. It remains to show convexity: let $\alpha\in[0;1]$, then
	\begin{multline*}
		\mu(\mu_{\Omega_1}(\alpha \xi_1+(1-\alpha)\tilde \xi_1,\alpha \eta_1+(1-\alpha)\tilde \eta_1),\ldots,
			\mu_{\Omega_n}(\alpha \xi_n+(1-\alpha)\tilde \xi_n,\alpha \eta_n+(1-\alpha)\tilde \eta_n)) \le \\
		\le \mu(\alpha \mu_{\Omega_1}(\xi_1,\eta_1) + (1-\alpha)\mu_{\Omega_1}(\tilde \xi_1,\tilde \eta_1),\ldots,
			\alpha \mu_{\Omega_n}(\xi_n,\eta_n) + (1-\alpha)\mu_{\Omega_n}(\tilde \xi_n,\tilde \eta_n)) \le \\
		\le \alpha \mu(\mu_{\Omega_1}(\xi_1,\eta_1),\ldots, \mu_{\Omega_n}(\xi_n,\eta_n)) + 
			(1-\alpha) \mu(\mu_{\Omega_1}(\tilde \xi_1,\tilde \eta_1),\ldots, \mu_{\Omega_n}(\tilde \xi_n,\tilde \eta_n)).
	\end{multline*}
	Here the first inequality holds by convexity of $\mu_{\Omega_i}$ and monotonicity of $\mu$, and the second inequality holds by convexity of $\mu$.
\end{proof}

If $U$ satisfies the main assumption, then we are able to obtain explicit formulae for extremals in~\eqref{eq:main_problem} in terms of convex trigonometry (see Sec.~\ref{sec:explicit_formulae}). Obviously, there are plenty of sets that do not satisfy the main assumption. Nonetheless, the following important cases do satisfy this assumption.

\begin{enumerate}
	\item Balls in $L_p$ metric, $1\le p\le\infty$ (see Sec.~\ref{sec:lp})
	\item Convex hulls and sums of coordinate compact convex sets $\Omega_i\subset O\dot x_i\dot y_i\subset \Delta(\id)$, $0\in\Int\Omega_i$ (see Sec.~\ref{sec:convex_hull_and_direct_product} and Remark~\ref{rm:arbitrary_symplectic_planes}).
	\item Any sub--Riemannian problem on $\Hb{2n+1}$ can be reduced to the one satisfying the main assumption by an appropriate symplectic change of variables (see Sec.~\ref{sec:subriemannain}).
\end{enumerate}

\section{Introduction to convex trigonometry}
\label{sec:ct}

We start with a brief explanation of convex trigonometry which was introduced for the first time in~\cite{CT1}. This section presents shortly main definitions and formulae of convex trigonometry without any proofs, which can be found in~\cite{CT1}.

Let $\Omega\subset\R^2$ be a convex compact set and let $0\in\mathrm{int}\,\Omega$. The following definition of the functions $\cos_\Omega$ and $\sin_\Omega$ at first glance may cause a natural question ``why so?''. Nonetheless, exactly this particular definition appears to be very convenient in solving optimal control problems with 2--dimensional control in~$\Omega$. First, these functions were introduced in~\cite{CT1}, where geodesics on 5 sub--Finsler problems were found. In~\cite{CT2}, convex trigonometry were used to integrate a series of left--invariant sub--Finsler problems on all unimodular 3D Lie groups and some other problems (including Finsler geodesics on Lobachevsky plane).

Denote by $\mathbb{S}$ the area of set $\Omega$.

\begin{figure}[ht]
	\centering
	\begin{subfigure}[t]{0.45\textwidth}
		\includegraphics[width=\textwidth]{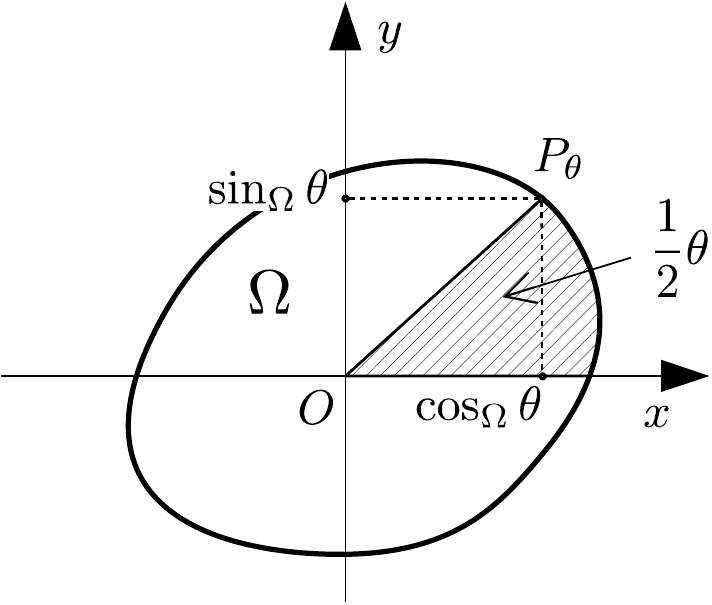}
		\subcaption{Definition of the generalized trigonometric functions $\cos_\Omega\theta$ and $\sin_\Omega\theta$ by set $\Omega$.}
		\label{fig:cos_sin_def}
	\end{subfigure}\hfill
	\begin{subfigure}[t]{0.45\textwidth}
		\includegraphics[width=\textwidth]{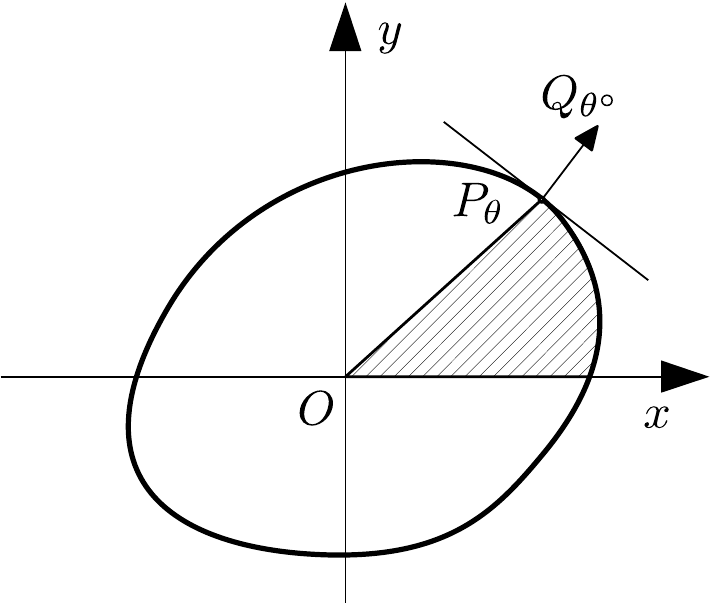}
		\subcaption{Correspondence $\theta\xleftrightarrow{\Omega\,}\theta^\polar$.}
		\label{fig:thetha_to_theta_polar_pic}
	\end{subfigure}
	\caption{Main convex trigonometry definitions.}
\end{figure}

\begin{definition}
	Let $\theta\in\R$ denote a generalized angle. If $0\le\theta<2\mathbb{S}$, then we choose a point $P_\theta$ on the boundary of $\Omega$ such that the area of the sector of $\Omega$ between the rays $Ox$ and $OP_\theta$ is $\frac12\theta$ (see Fig.~\ref{fig:cos_sin_def}). By definition $\cos_\Omega\theta$ and $\sin_\Omega\theta$ are the coordinates of~$P_\theta$. If the generalized angle $\theta$ does not belong to the interval $\big[0;2\mathbb{S}\big)$, then we define the functions $\cos_\Omega$ and $\sin_\Omega$ as periodic with the period $2\mathbb{S}$; i.e., for $k\in\Z$ such that $\theta + 2\mathbb{S}k \in[0;2\mathbb{S})$ we put
	\[
	\cos_\Omega\theta = \cos_\Omega(\theta + 2\mathbb{S}k);\qquad
	\sin_\Omega\theta = \sin_\Omega(\theta + 2\mathbb{S}k);\qquad
	P_\theta = P_{\theta+2\mathbb{S}k}.
	\]
\end{definition}

Note that all the properties of $\sin_\Omega$ and $\cos_\Omega$ listed below can be easily proved once the appropriate definition is given. 

Obviously, $\sin_\Omega0=0$. If $\Omega$ is the unit circle centered at the origin, then the above definition produces the classical trigonometric functions. If $\Omega$ differs from the unit circle, then the functions $\cos_\Omega$ and $\sin_\Omega$, of course, differ from the classical functions $\cos$ and $\sin$. Nonetheless they inherit a lot of properties from the classical case and can be usually computed explicitly.

We will use the polar set $\Omega^\polar$ together with the set $\Omega$:
\[
\Omega^\polar = \{(p,q)\in\R^{2*}:px+qy\le 1\mbox{ for all }(x,y)\in\Omega\}\subset\R^{2*}.
\]

\noindent The polar set $\Omega^\polar$ is (always) a convex and compact (as $0\in\mathrm{int}\,\Omega$) set and $0\in\mathrm{int}\,\Omega^\polar$ (as $\Omega$ is bounded). To avoid confusion we will assume that the set $\Omega$ lies in the plane with coordinates $(x,y)$ and the polar set $\Omega^\polar$ lies in the plane with coordinates $(p,q)$.

Note that $\Omega^{\circ\circ}=\Omega$ by the bipolar theorem (see \cite[Theorem 14.5]{Rockafellar}). We can apply the above definition of the generalized trigonometric functions to the polar set $\Omega^\polar$ and an arbitrary angle $\psi\in\R$ to construct $\cos_{\Omega^\polar}\psi$ and $\sin_{\Omega^\polar}\psi$, which are the coordinates of the appropriate point $Q_\psi\in\partial\Omega^\polar$. From the definition of the polar set it follows that
\begin{equation}
\label{eq:non_pythagorean_inequality}
	\cos_\Omega\theta\cos_{\Omega^\polar}\psi + 
	\sin_\Omega\theta\sin_{\Omega^\polar}\psi \le 1.
\end{equation}

\begin{definition}
	We say that angles $\theta\in\R$ and $\theta^\polar\in\R$ correspond to each other and write $\theta\xleftrightarrow{\Omega\,}\theta^\polar$ if the supporting half-plane of $\Omega$ at $P_\theta$ is determined by the (co)vector~$Q_{\theta^\polar}$ (see Fig.~\ref{fig:thetha_to_theta_polar_pic}).
\end{definition}

As it was said properties of the classical functions $\cos$ and $\sin$ are inherited by two pairs of functions for the sets $\Omega$ and $\Omega^\polar$. We start with the Pythagorean identity $\cos^2\theta+\sin^2\theta=1$, which takes the following form:

\begin{theorem}[see~{\cite[Theorem~1]{CT1}}]
	\label{thm:main_trig_thm}
	The definition of the correspondence of $\theta$ and $\theta^\polar$ is symmetric, i.e., $\theta\xleftrightarrow{\Omega\,}\theta^\polar$ is equivalent to $\theta^\polar\xleftrightarrow{\Omega^\polar}\theta$. Moreover, an analogue of the main Pythagorean identity holds:
	\begin{equation}
	\label{eq:main_trig}
	\theta\xleftrightarrow{\Omega\,}\theta^\polar
	\quad\Longleftrightarrow\quad
	\cos_\Omega\theta\cos_{\Omega^\polar}\theta^\polar + 
	\sin_\Omega\theta\sin_{\Omega^\polar}\theta^\polar = 1.
	\end{equation}
\end{theorem}

The correspondence $\theta\xleftrightarrow{\Omega\,}\theta^\polar$ is not one-to-one in general. If the boundary of $\Omega$ has a corner at a point $P_\theta$, then the angle $\theta$ corresponds to the whole edge in $\Omega^\polar$ and vice versa, i.e., to any angle $\theta$ with $P_\theta$ on the same edge of $\Omega$ there corresponds one particular angle $\theta^\polar$ (up to $2\mathbb{S}^\polar\Z$, where $\mathbb{S}^\polar$ denotes the area of $\Omega^\polar$), and the boundary of $\Omega^\polar$ has a corner at the point $Q_{\theta^\polar}$. Nonetheless, it is natural to define a monotonic (multivalued and closed) function $\theta^\polar(\theta)$ that maps an angle $\theta$ to a maximal closed interval\footnote{Obviously, for any $\theta$ there exists an angle $\theta^\polar$ such that $\theta^\polar\xleftrightarrow{\Omega^\polar}\theta$. This can be easily proved by the hyperplane separation theorem.} of angles $\theta^\polar$ such that $\theta^\polar\xleftrightarrow{\Omega^\polar}\theta$. This function is quasiperiodic, i.e.,
\[
\theta^\polar(\theta+2\mathbb{S}k) = \theta^\polar(\theta) + 2\mathbb{S}^\polar k
\quad\mbox{with}\quad k\in\Z.
\]

If $\Omega$ is strictly convex, then the function $\theta^\polar(\theta)$ is strictly monotonic. If the boundary of $\Omega$ is $C^1$--smooth, then the function $\theta^\polar(\theta)$ is continuous.

Let us now compute derivatives of the functions $\cos_\Omega$ and $\sin_\Omega$. In the classical case $\cos\theta$ and $\sin\theta$ are smooth, $\cos'\theta=-\sin\theta$ and $\sin'\theta=\cos\theta$. In general case $\cos_\Omega\theta$ and $\sin_\Omega\theta$ are Lipschitz continuous and their derivatives are $-\sin_{\Omega^\polar}\theta^\polar$ and $\cos_{\Omega^\polar}\theta^\polar$. Precisely

\begin{theorem}[see~{\cite[Theorem~2]{CT1}}]
	\label{thm:derivatives}
	The functions $\cos_\Omega$ and $\sin_\Omega$ are Lipschitz continuous and have the left and right derivatives for all $\theta$, which coincide for a.e.\ $\theta$. Let us denote for short the whole interval between the left and right derivatives by the usual derivative stroke sign (if this set contains only one element, we usually omit braces). Then for a.e.~$\theta$, we have
	\[
	\cos'_\Omega\theta = -\sin_{\Omega^\polar}\theta^\polar
	\quad\mbox{ and }\quad
	\sin'_\Omega\theta = \cos_{\Omega^\polar}\theta^\polar,
	\]
	
	\noindent where $\theta\xleftrightarrow{\Omega\,}\theta^\polar$. Moreover, for any $\theta$
	\begin{align*}
	&\cos'_\Omega\theta = \{-\sin_{\Omega^\polar}\theta^\polar\quad\mbox{for all}\quad \theta^\polar\xleftrightarrow{\Omega^\polar}\theta\},\\
	&\sin'_\Omega\theta = \{\cos_{\Omega^\polar}\theta^\polar\quad\mbox{for all}\quad \theta^\polar\xleftrightarrow{\Omega^\polar}\theta\}.
	\end{align*}
	
	\noindent The similar formulae hold for $\cos_{\Omega^\polar}'\theta^\polar$ and $\sin_{\Omega^\polar}'\theta^\polar$.
\end{theorem}

The two types of formulae for derivatives stated in the previous theorem coincide if for given $\theta$ there exists a unique $\theta^\polar\xleftrightarrow{\Omega^\polar}\theta$. If so, then the both functions $\cos_\Omega$ and $\sin_\Omega$ have derivatives at $\theta$. Precisely, the function $\cos_\Omega$ has derivative at $\theta$ iff values of $\sin_{\Omega^\polar}\theta^\polar$ coincide for all $\theta^\polar\xleftrightarrow{\Omega^\polar}\theta$, and uniqueness of $\theta^\polar\xleftrightarrow{\Omega^\polar}\theta$ is an obvious sufficient condition for this. The function $\sin_\Omega$ has a similar property.

Let us note that any Lipschitz continuous function is a.e.\ differentiable. So if no confuse ensues we will write for short $\cos_\Omega'\theta=-\sin_{\Omega^\polar}\theta^\polar$ and $\sin_\Omega'\theta=\cos_{\Omega^\polar}\theta^\polar$ always meaning the result obtained in Theorem~\ref{thm:derivatives}.

It is easy to see that both functions $\cos_\Omega$ and $\sin_\Omega$ have one interval of increasing and one interval of decreasing during their period. These two intervals can be separated by at most two intervals of constancy, which appear if $\Omega$ has edges parallel to the axes. Intervals of convexity and concavity can be also determined by the formulae of differentiation.

\begin{corollary}[see~{\cite[Corollary~1]{CT1}}]
	Each of the functions $\cos_\Omega$ and $\sin_\Omega$ is concave on any interval with non-positive values and is convex on any interval with non-negative values.
\end{corollary}

We also need an analogue of the polar change of coordinates:
\begin{equation}
\label{eq:polar_change_of_coordinates}
	\begin{cases}
		x=r\cos_\Omega\theta;\\
		y=r\sin_\Omega\theta.
	\end{cases}
\end{equation}

\noindent Note that in the classical case the angles are defined up to a summand $2\pi k$, $k\in\Z$. Here we have a similar situation: generalized angles are defined up to a summand $2\mathbb{S}k$, $k\in\Z$.

This change of variables is smooth in $r$ and Lipschitz continuous in $\theta$. Hence it has a.e.\ partial derivative with respect to $\theta$. The Jacobian matrix has the following form:
\[
J=\begin{pmatrix}
x'_r & x'_\theta\\
y'_r & y'_\theta
\end{pmatrix}
=
\begin{pmatrix}
\cos_\Omega\theta & -r\sin_{\Omega^\polar}\theta^\polar\\
\sin_\Omega\theta & r\cos_{\Omega^\polar}\theta^\polar\\
\end{pmatrix},
\quad\mbox{where}\quad
\theta^\polar \xleftrightarrow{\Omega^\polar}\theta.
\]

\noindent Using the main Pythagorean identity we see that the Jacobian is equal to~$r$:
\[
	\det J = r.
\]

Let us find the inverse change of variables $r(x,y)$ and $\theta(x,y)$. The most convenient way to do this is the following one:

\begin{theorem}[{\cite[Section 2]{CT1} and \cite[Theorem 3]{CT2}}]
\label{thm:polar_change}
	Let $(x(t),y(t))$ be an absolutely continuous curve that does not pass through the origin. Then the functions $r(t)$ and $\theta(t)$ from \eqref{eq:polar_change_of_coordinates} are absolutely continuous\footnote{The angle $\theta$ is defined up to $2\mathbb{S}\Z$ as always.} and satisfy
	\[
		r=s_{\Omega^\polar}(x,y)\qquad\mbox{and}\qquad
		\dot\theta=\frac{x\dot y-\dot x y}{r^2}.
	\]
	\noindent The first equation holds for all $t$, and the second one holds for a.e.\ $t$.
\end{theorem}

In~\cite{CT1} using these formulae, the functions $\cos_\Omega$ and $\sin_\Omega$ were completely computed for the case when $\Omega$ is an arbitrary polygon. Some additional examples of $\Omega$ were computed in~\cite[Exmples 1,2, and 3]{CT2}. In this paper we compute $\cos_\Omega$ and $\sin_\Omega$ for $\Omega$ being the unit balls on $\R^2$ in $L_p$ metric for $1<p<\infty$ (see Sec.~\ref{sec:lp}).

\section{Explicit formulae in terms of convex trigonometry}
\label{sec:explicit_formulae}

In this section, we obtain explicit formulae for extremals in problem~\eqref{eq:main_problem} in the case when set $U$ of admissible controls satisfies the main assumption.

We denote by $\partial s(A)$ the subdifferential of a convex function $s$ at a point $A$ as usual. We will also use the following compact convex set
\[
	\Xi=\{(\lambda_1,\ldots,\lambda_n):\lambda_i\ge 0,\mu(\lambda_1,\ldots,\lambda_n)\le 1\},
\]
whose supporting function is denoted by $s_\Xi$.
\begin{theorem}
\label{thm:main}
	Suppose that $U$ satisfies the main assumption. Then for any extremal\footnote{I.e.\ a solution to PMP projection on the base $\Hb{2n+1}$.} $(x(t),y(t),z(t))$ in problem~\eqref{eq:main_problem} on $\Hb{2n+1}$, there exists
	\begin{itemize}
		\item constants $\gamma=0,\pm1$ and $A=(A_1,\ldots,A_n)\in\R^n_+$ not vanishing at the same time; 
		\item constants $\theta_{i0}^\polar\in\R$ for each $i$ with $A_i>0$; 
		\item measurable functions $\lambda(t)=(\lambda_1(t),\ldots,\lambda_n(t))\in\partial s_{\Xi}(A)$ satisfying the property
		\[
			\forall i\quad(A_i=0\,\&\,\gamma\ne 0)\Rightarrow \lambda_i\equiv0
		\]
	\end{itemize}
	such that
	\begin{enumerate}
		\item\label{item:gamma_not_null} If $\gamma\ne 0$, then for each $i$ with $A_i>0$, we have
		\[
			x_i=\gamma A_i(\sin_{\Omega_i^\polar}\theta_i^\polar - \sin_{\Omega_i^\polar}\theta_{i0}^\polar)
			\quad\mbox{and}\quad
			y_i=\gamma A_i(\cos_{\Omega_i^\polar}\theta_{i0}^\polar - \cos_{\Omega_i^\polar}\theta_i^\polar)
		\]
		where $\theta_i^\polar = \theta_{i0}^\polar + \frac{\gamma}{A_i} \int_0^t\lambda_i(\tau)\,d\tau$; for each $i$ with $A_i=0$, we have $x_i\equiv y_i\equiv0$; and
		\[
			2z=	\gamma s_\Xi(A) t +
					\sum_{i:A_i>0} A_i^2\left(
						\sin_{\Omega_i^\polar}\theta_{i0}^\polar\cos_{\Omega_i^\polar}\theta_i^\polar -
						\cos_{\Omega_i^\polar}\theta_{i0}^\polar\sin_{\Omega_i^\polar}\theta_i^\polar
					\right).
		\]
		
		\item If $\gamma=0$, then for each $i$, we have
		\[
			x_i = \int_0^t \lambda_i(\tau)\cos_{\Omega_i}\theta_i (\tau)\,d\tau
			\quad\mbox{and}\quad
			y_i = \int_0^t \lambda_i(\tau)\sin_{\Omega_i}\theta_i (\tau)\,d\tau
		\]
		where $\theta_i(t)$ is a measurable function such that $\theta_i(t)\xleftrightarrow{\Omega_i}\theta_{i0}^\polar$ if $A_i>0$, and $\theta_i(t)$ is an arbitrary measurable function if $A_i=0$; and\footnote{Since $\lambda_i(t)$ and $\theta_i(t)$ have a lot of freedom, nothing more can be said about $x$, $y$, and $z$ in the general case $\gamma=0$ (this situation is completely similar to the case $n=1$, see~\cite{BerestovskiiHeisenberg,CT1}).} $z=\frac12\int_0^t \sum_i (x_i(\tau)\dot y_i(\tau)-\dot x_i(\tau)y_i(\tau))\,d\tau$.
	\end{enumerate}
	Moreover, if a trajectory $(x(t),y(t),z(t))$ has one of the described forms, then it is an extremal in problem~\eqref{eq:main_problem} on $\Hb{2n+1}$.
\end{theorem}

We start with some discussion on results of the theorem and after that present the proof.

The case $\gamma\ne 0$ has very nice geometrical interpretation: if $A_i>0$ for some $i$, then pair $(x_i,y_i)$ moves along the boundary of the polar set $\Omega_i^\polar$ rotated $-\gamma90^\polar$ (remind $\gamma=\pm1$), stretched by $A_i$ times, and shifted in such a way that the origin belongs to the obtained set boundary. Moreover, all this rotations have the same direction: counterclockwise if $\gamma=1$ or clockwise if $\gamma=-1$. The motion speed is determined by $\lambda_i(t)$, which may vary in time if function $\mu$ is not strictly convex.

\begin{corollary}
\label{cor:stricly_convex}
	If $\mu$ is strictly convex, then $\lambda$ is constant and functions $\theta_i^\polar$ in the case $\gamma\ne 0$ of Theorem~\ref{thm:main} becomes linear, $\theta_i^\polar = \theta_{i0}^\polar + \gamma\lambda_it/A_i$. If additionally set $\Omega_i$ is strictly convex for some index $i$, then $x_i$ and $y_i$ in the case $\gamma=0$ of Theorem~\ref{thm:main} are linear.
\end{corollary}

\begin{proof}
	If $\mu$ is strictly convex, then $s_{\Xi}$ is $C^1(\R^n_+)$ (see~\cite[Theorem~25.1]{Rockafellar}), so for any $A\in\R^n_+$, $\partial s_{\Xi}(A)$ consists of a unique element $s'_\Xi(A)$, which must coincide with $\lambda$. Moreover, since $\mu$ is strictly convex and monotone, it must be strictly monotone. Hence if $A_i=0$ then $\lambda_i=0$ even if $\gamma=0$, since $\partial s_{\Xi}(A) = \{\lambda\in\Xi:\forall\tilde\lambda\in\Xi\ \sum_iA_i\lambda_i \ge \sum_iA_i\tilde\lambda_i\}$. 
	
	Suppose additionally that $\Omega_i$ is strictly convex for some index $i$. If $\gamma=0$ and $A_i>0$, then $x_i$ and $y_i$ are determined by $\theta_i(t)\xleftrightarrow{\Omega_i}\theta_{i0}^\polar$, but there exists a unique angle corresponding to $\theta_{i0}^\polar$ w.r.t. $\Omega_i$. Hence $\theta_i(t)=\const$. If $\gamma=A_i=0$, then $\lambda_i\equiv 0$ as was shown. Hence in the case $\gamma=0$, $x_i$ and $y_i$ are linear.
\end{proof}

So, if $\mu$ is strictly convex, then in the case $\gamma\ne 0$ each pair $(x_i(t),y_i(t))$ satisfies Kepler's law: radius vector $(x_i(t),y_i(t))$ on the plane $Ox_iy_i$ sweeps out equal areas during equal intervals of time by corollary~\ref{cor:stricly_convex}. This law is always fulfilled on $\Hb{3}$ in the case $\gamma\ne 0$ again by corollary~\ref{cor:stricly_convex}.

Note that if $\mu$ is not strictly convex, then $\partial s_\Xi$ still consists of 1 element for a.e.~$A$ (see~\cite[Theorem~25.5]{Rockafellar}), and for those $A$, functions $\lambda_i$ are constants and Kepler's Law is fulfilled in the case~$\gamma\ne 0$.	

\begin{corollary}
\label{cor:mu_strictly_monotone}
	Suppose that function $\mu$ is strictly monotone in $\lambda_i$ for some index $i$. Then if $A_i=0$, then~$x_i\equiv y_i\equiv0$.
\end{corollary}

\begin{proof}
	Indeed, if $A_i=0$ then $\lambda_i=0$ by $\mu$ strictly monotonicity in $\lambda_i$ as it was shown. Hence $x_i\equiv y_i\equiv0$ even in the case $\gamma=0$.
\end{proof}

Since $\lambda\in\partial s_{\Xi}(A)$ in Theorem~\ref{thm:main}, it is nice to have a convenient way for computation $\partial s_\Xi(A)$.

\begin{remark}
\label{rm:subdiff_is_argmax}
	From the definition of support functions it follows that
	\[
		\partial s_\Xi(A)=\argmax_{\lambda\in\Xi}\sum_{i=1}^n A_i\lambda_i,
	\]
	and hence, $\lambda\in\Xi$. In particular, $\lambda_i\ge 0$.
\end{remark}

Now, we are ready to prove Theorem~\ref{thm:main}.

\begin{proof}[Proof of Theorem~\ref{thm:main}]

First, let us write down the control system in coordinates:
\begin{equation}
\label{eq:five_dim_control_system}
	\dot x_i=\lambda_i u_i;\quad \dot y_i=\lambda_iv_i;\quad 
	\dot z=\frac12\sum_{i=1}^n \lambda_i(x_iv_i - y_iu_i).
\end{equation}
where $i=1,\ldots,n$ and
\[
	(u_i,v_i)\in\partial\Omega_i\qquad\mbox{and}\qquad
	(\lambda_1,\ldots,\lambda_n)\in\Xi
\]
are controls. So we have parametrized $2n$--dimesional set $U$ by $2n$ parameters, since $\partial\Omega_i$ are 1-dimensional sets. Each point in $U$ can be written in the described form (by putting $\lambda_i=\mu_{\Omega_i}(\dot x_i,\dot y_i)$), but if $\dot x_i=\dot y_i=0$ for some $i$, then this form is not unique.

Let us write down the Pontryagin function (Hamiltonian) in the time minimization problem for control system~\eqref{eq:five_dim_control_system}:
\[
	\mathcal{H} = \sum_{i=1}^n\lambda_i\left(\varphi_i u_i + \psi_iv_i + \frac\gamma2 (x_i v_i-y_iu_i)\right)
\]
where $\varphi_i$ are conjugate to $x_i$, $\psi_i$ are conjugate to $y_i$, and $\gamma$ is conjugate to $z$, and all conjugate variables are not allowed to vanish simultaneously. Following traditions of sub--Riemannian geometry, we denote coefficients at control variables by\footnote{Obviously, $h_i$, $g_i$, and $\gamma$ are linear on fibers left--invariant coordinates on $T^*\Hb{2n+1}$.}
\[
	h_i = \varphi_i-\frac12\gamma y_i\quad\mbox{and}\quad g_i = \psi_i + \frac12\gamma x_i,\qquad i=1,\ldots,n.
\]
Then from the Hamiltonian equations for $\mathcal{H}$, we obtain
\begin{equation}
\label{eq:hg_dim_vertical_subsys}
	\dot h_i = -\gamma\lambda_iv_i;\qquad \dot g_i = \gamma\lambda_i u_i;\qquad \dot\gamma=0.
\end{equation}

These equations do not depend on the structure of $U$ and form so called vertical subsystem of PMP. But, according to PMP, optimal control maximize $\mathcal{H}$ for a.e.\ $t$ among all admissible controls:
\begin{equation}
\label{eq:five_dim_control_max}
	\sum_{i=1}^n (h_i\lambda_iu_i+g_i\lambda_iv_i) \to \max_{(u_i,v_i)\in\partial\Omega_i;\ \lambda\in\Xi}
\end{equation}
Solution to this maximization problem highly depends on the structure of set $U$ boundary. We are able to solve equations~\eqref{eq:hg_dim_vertical_subsys}, \eqref{eq:five_dim_control_max} via machinery of convex trigonometry precisely because $U$ satisfies the main assumption. 

Since $(u_i,v_i)\in\partial\Omega_i$, we have
\[
	u_i=\cos_{\Omega_i}\theta_i\qquad\mbox{and}\qquad v_i=\sin_{\Omega_i}\theta_i.
\]
for some $\theta_i$. Put
\[
	h_i = A_i \cos_{\Omega_i^\polar}\theta_i^\polar\qquad\mbox{and}\qquad g_i= A_i\sin_{\Omega_i^\polar}\theta_i^\polar
\]
where $A_i=s_{\Omega_i}(h_i,g_i)\ge 0$ (since $0\in\Omega_i$), and $\theta_i^\polar$ is well defined iff $A_i>0$.

Hence, if $A_i>0$ and $\lambda_i>0$, then from~\eqref{eq:five_dim_control_max}, we have $\theta_i\stackrel{\Omega_i}{\rightarrow}\theta_i^\polar$ by inequality~\eqref{eq:non_pythagorean_inequality} and the generalized Pythagorean identity (see Theorem~\ref{thm:main_trig_thm}). If $\lambda_i=0$, then the choice of the pair $(u_i,v_i)$ does not change $\dot x_i$ and $\dot y_i$, so in this case, we may also assume that if $A_i>0$, then $\theta_i\stackrel{\Omega_i}{\rightarrow}\theta_i^\polar$.

\begin{proposition}
	On any extremal, we have $A=(A_1,\ldots,A_n)=\const$.
\end{proposition}

\begin{proof}

	Fix an index $i$. First, suppose that $A_i(t_0)>0$ at some instant $t_0$. Then in a neighborhood of $t_0$, using Theorem~\ref{thm:polar_change}, we obtain
	\begin{equation}
	\label{eq:dot_theta_i}
		\dot \theta_i^\polar = \frac{h_i\dot g_i - \dot h_i g_i}{A_i^2} = 
		\frac{A_i\cos_{\Omega_i^\polar}\theta_i^\polar\gamma\lambda_i\cos_{\Omega_i}\theta_i + A_i\sin_{\Omega_i^\polar}\theta_i^\polar\gamma\lambda_i\sin_{\Omega_i}\theta_i}{A_i^2} = 
		\frac{\gamma\lambda_i}{A_i}.
	\end{equation}
	The last equation holds by the generalized Pythagorean identity (see Theorem~\ref{thm:main_trig_thm}). Thus, using formulae for $\cos_\Omega$ and $\sin_\Omega$ derivatives (see Theorem~\ref{thm:derivatives}), we obtain
	\[
		\left\{\begin{array}{l}
			\dot h_i = \dot A_i \cos_{\Omega_i^\polar}\theta_i^\polar - \gamma\lambda_i \sin_{\Omega_i}\theta_i;\\
			\dot g_i = \dot A_i \sin_{\Omega_i^\polar}\theta_i^\polar + \gamma\lambda_i \cos_{\Omega_i}\theta_i.\\
		\end{array}\right.
	\]
	Hence, using~\eqref{eq:hg_dim_vertical_subsys}, we obtain
	\[
		\dot A_i = \dot h_i u_i + \dot g_i v_i = 0.
	\]
	
	So, we have proved, that function $A_i(t)$ is locally constant on the open set $\{t:A_i(t)>0\}$. Since $A_i$ is continuous, then it must be constant for all t.
	
\end{proof}

\begin{corollary}
	If $A_i>0$, then $\theta_i^\polar(t) = \theta_{i0}^\polar+\frac{\gamma}{A_i}\int_0^t\lambda_i(\tau)\,d\tau$ for some constant $\theta_i^\polar$.
\end{corollary}

\begin{proof}
	This follows immediately from~\eqref{eq:dot_theta_i}, since $A_i$ and $\gamma$ are constants.
\end{proof}

\begin{corollary}
	Constants $\gamma$ and $A$ do not vanish simultaneously.
\end{corollary}

\begin{proof}
	Indeed, if $\gamma=A_i=0$ for all $i$, then $h_i\equiv g_i\equiv 0$. Hence $\varphi_i\equiv\psi_i\equiv0$. So, all Lagrange multipliers $\varphi_i$, $\psi_i$, and $\gamma$ vanish simultaneously, which is forbidden by PMP.
\end{proof}

So, on any extremal, we have
\[
	h_iu_i + g_iv_i \equiv A_i = \const.
\]
Hence, $\lambda$ at any instant $t$ is a solution to the following time-independent maximization problem
\begin{equation}
\label{eq:max_A_i_lambda_i}
	\sum_{i=1}^n A_i\lambda_i \to \max_{\lambda\in\Xi}.
\end{equation}
Note that $\lambda$ is a solution to~\eqref{eq:max_A_i_lambda_i} iff $\lambda\in\partial s_\Xi(A)$ (see Remark~\ref{rm:subdiff_is_argmax}). Since $A_i\ge 0$, maximum in~\eqref{eq:max_A_i_lambda_i} is always attained in a point on the convex surface $\partial\Xi_+=\{\lambda_i\ge 0:\mu(\lambda)=1\}\subset\partial\Xi$.

Since $A_i$ are constants, if function $\mu$ is strictly convex, then $\lambda$ is also constant. If function $\mu$ is not strictly convex, then $A$ may define a support hyperplane to $\Xi$ at a face $F\subset\partial\Xi_+$. In this case, $\lambda(t)\in F$ is an arbitrary measurable function. This phenomenon does not appear in the smallest dimension $n=1$ (since if $n=1$, then $\dim F=0$). Surprisingly, even if $\dim F\ne 0$, we are able to completely integrate equations on $x$, $y$, and $z$ in the case $\gamma\ne0$ despite arbitrariness in the choice of $\lambda(t)$.

First, consider the simplest case $\gamma=0$.

\begin{enumerate}
	\item\label{item:zn} Let $A_i>0$. In this case, $\dot\theta_i^\polar\equiv 0$. Hence, choosing an arbitrary measurable function $\theta_i(t)\xleftrightarrow{\Omega_i}\theta_i^\polar$ we obtain admissible controls $u_i(t)=\cos_{\Omega_i}\theta_i(t)$ and $v_i(t)=\sin_{\Omega_i}\theta_i(t)$. 
	
	\item\label{item:zz} Let $A_i=0$. In this case $(u_i(t),v_i(t))\in\partial\Omega_i$ can be chosen arbitrary.
\end{enumerate}

So we have proved item~\ref{item:gamma_not_null} of the theorem. Moreover, any above constructed trajectory with $\gamma=0$ is an extremal with $h_i=A_i\cos_{\Omega_i^\polar}\theta_i^\polar=\const$ and $g_i=A_i\sin_{\Omega_i^\polar}\theta_i^\polar=\const$.

Now consider the interesting case $\gamma\ne 0$. Let us now find $x_i$ and $y_i$ explicitly.

\begin{enumerate}
	\item\label{item:nn} Let $A_i>0$. In this case, $\lambda_i=A_i\dot\theta_i^\polar/\gamma$. Therefore,
	\[
		\dot x_i = \lambda_i\cos_{\Omega_i}\theta_i = \frac{A_i}{\gamma}\dot\theta_i^\polar\cos_{\Omega_i}\theta_i =
		\frac{d}{dt}\left(\frac{A_i}{\gamma}\sin_{\Omega_i^\polar}\theta_i^\polar\right),
	\]
	and $x_i=x_i^0 + \frac{A_i}{\gamma}\sin_{\Omega_i^\polar}\theta_i^\polar$, where $x_i^0$ is a constant. Similarly, $y_i=y_i^0 - \frac{A_i}{\gamma}\cos_{\Omega_i^\polar}\theta_i^\polar$. Using initial conditions $x_i(0)=y_i(0)=0$, we get $x_i^0=-\frac{A_i}{\gamma}\sin_{\Omega_i^\polar}\theta_{i0}^\polar$ and $y_i^0=\frac{A_i}{\gamma}\cos_{\Omega_i^\polar}\theta_{i0}^\polar$. Without loss of generality, $\gamma=\pm1$, since Lagrange multipliers are defined up to multiplication by a positive constant. Hence $\gamma=1/\gamma$, and we have obtained formulae for $x_i$ and $y_i$ in item~\ref{item:gamma_not_null} of the theorem for the case $A_i>0$.
	
	\item\label{item:nz} Let $A_i=0$. In this case, $\theta_i^\polar$ is not well defined, but $h_i\equiv g_i\equiv0$. Since $\dot h_i=-\gamma\lambda_iu_i$ and $\dot g_i=\gamma\lambda_iv_i$, we obtain $\lambda_i\equiv0$. Hence, $\dot x_i\equiv\dot y_i\equiv0$, and $x_i\equiv y_i\equiv 0$, since $x_i(0)=y_i(0)=0$.
\end{enumerate}

In the case $\gamma\ne 0$, we are also able to find $z$ explicitly. Indeed, if $A_i>0$ then
\begin{multline*}
	\lambda_i(x_iv_i - y_iu_i) = 
		\lambda_i\sin_{\Omega_i}\theta_i (x_i^0+\frac{A_i}{\gamma}\sin_{\Omega_i^\polar}\theta_i^\polar) -
		\lambda_i\cos_{\Omega_i}\theta_i (y_i^0-\frac{A_i}{\gamma}\cos_{\Omega_i^\polar}\theta_i^\polar) =\\
	=\frac{A_i}{\gamma}\lambda_i + x_i^0\lambda_iv_i - y_i^0\lambda_iu_i =
	\frac{d}{dt}\left(\frac{A_i^2}{\gamma^2}\theta_i^\polar + x_i^0y_i - y_i^0x_i\right).
\end{multline*}

Hence, if $\gamma\ne 0$, we have
\[
	z=z^0 +	\frac12\sum_{i:A_i>0}\frac{A_i}{\gamma}\left(
			\frac{A_i}{\gamma}\theta_i^\polar - 
			y_i\sin_{\Omega_i^\polar}\theta_{i0}^\polar  -
			x_i\cos_{\Omega_i^\polar}\theta_{i0}
		\right).
\]
where $z^0=-\frac{1}{2\gamma^2}\sum_{i:A_i>0}A_i^2\theta_{i0}^\polar$, since $x_i(0)=y_i(0)=z(0)=0$.

Let us compute the following term in $z$:
\[
	\frac12\sum_{i:A_i>0}\frac{A_i^2}{\gamma^2}(\theta_i^\polar-\theta_{i0}^\polar) =
	\frac12\sum_{i=1}^n\frac{A_i}{\gamma}\int_0^t\lambda_i(\tau)\,d\tau = 
	\frac1{2\gamma}\int_0^t \sum_{i=1}^n A_i\lambda_i(\tau)\,d\tau = 
	\frac1{2\gamma}\int_0^t s_\Xi(A)\,dt = \frac{1}{2\gamma}s_\Xi(A)t,
\]
since $A=\const$. Putting $\gamma=\pm 1$, we prove item~\ref{item:gamma_not_null} of the theorem. 

Moreover, any above constructed trajectory with $\gamma\ne 0$ is an extremal with $h_i=A_i\cos_{\Omega_i^\polar}\theta_i^\polar$ and $g_i=A_i\sin_{\Omega_i^\polar}\theta_i^\polar$. Let us prove this. If $A_i>0$, then $\dot h_i = A_i\dot\theta_i^\polar\sin_{\Omega_i}\theta_i=\gamma\lambda_iu_i$ for a.e.\ $t$, since if $\dot\theta_i^\polar(t_0)\ne 0$ for some $t_0$, then for a.e.~$t$ in a neighborhood of $t_0$, there exists a unique $\theta(t)\xleftrightarrow{\Omega_i}\theta_i^\polar(t)$; and if $\dot\theta_i^\polar(t_0)=0$, then $\lambda_i(t_0)=0$. If $A_i=0$, then $h_i\equiv 0$ and $\gamma\lambda_iu_i=0$, since $\lambda_i$ was chosen be be null in the case $A_i=0$. Hence $h_i$ satisfies the Hamiltonian equations. Similarly, $g_i$ satisfies the Hamiltonian equations. Functions $x_i$, $y_i$, and $z$ satisfy control system and the initial conditions by construction. It remains to say, that $\lambda$ was chosen to maximize $\mathcal{H}=\sum_i A_i\lambda_i$, and pair $(u_i,v_i)\in\partial\Omega_i$ was chosen to maximize $h_iu_i+g_iv_i$, and $A_i=\max_{(u_i,v_i)\in\partial\Omega_i}(h_iu_i+g_iv_i)$. Q.E.D.

\end{proof}

\section{Case \texorpdfstring{$L_p$}{Lp} for \texorpdfstring{$1<p<\infty$}{1<p<∞}}
\label{sec:lp}

In this section, we demonstrate exact formulae in the case $U=\{(\dot x,\dot y)\in\R^n\times\R^{n*}:|\dot x|^p + |\dot y|^p\le 1\}\subset\R^n\times\R^{n*}$ where $1<p<\infty$ in terms of incomplete Euler integral of the first kind (which can be expressed in terms of hypergeometric function $_2F_1$). The cases $p=1$ and $p=\infty$ are considered in Sec.~\ref{sec:convex_hull_and_direct_product}. The plane case $n=1$ is the most important one. Indeed, if we put $\Omega=\{(u,v)\in\R^2:|u|^p+|v|^p\le1\}\subset\R^2$ and $\mu(\lambda_1,\ldots,\lambda_n)=(\sum_i\lambda_i^p)^{1/p}$, then the main assumption is fulfilled, since
\[
	U=\left\{(\dot x,\dot y)\in\R^n\times\R^{n*}:\mu(\mu_\Omega(\dot x_1,\dot y_1),\ldots,\mu_\Omega(\dot x_n,\dot y_n))\le 1\right\}.
\]
Hence, solutions to this case can be written in term of $\cos_\Omega$ and $\sin_\Omega$ by Theorem~\ref{thm:main}.

We start with computing functions of convex trigonometry for $\Omega$ being the unit ball in $L_p$ metric. In this case, polar set $\Omega^\polar$ is the unit ball in $L_q$ metric, where $pq=p+q$ (see Fig.~\ref{fig:lp_lp_balls}). Let us parametrize $\partial\Omega$ in the following way
\[
	\partial\Omega = \{(u,v):\ u=|\cos\varphi|^{2/p}\sgn{\cos\varphi},\ v=|\sin\varphi|^{2/p}\sgn{\sin\varphi},\quad \varphi\in\R\}.
\]

\begin{figure}[ht]
	\centering
	\begin{subfigure}{0.3\textwidth}
		\centering
		\includegraphics[width=\textwidth]{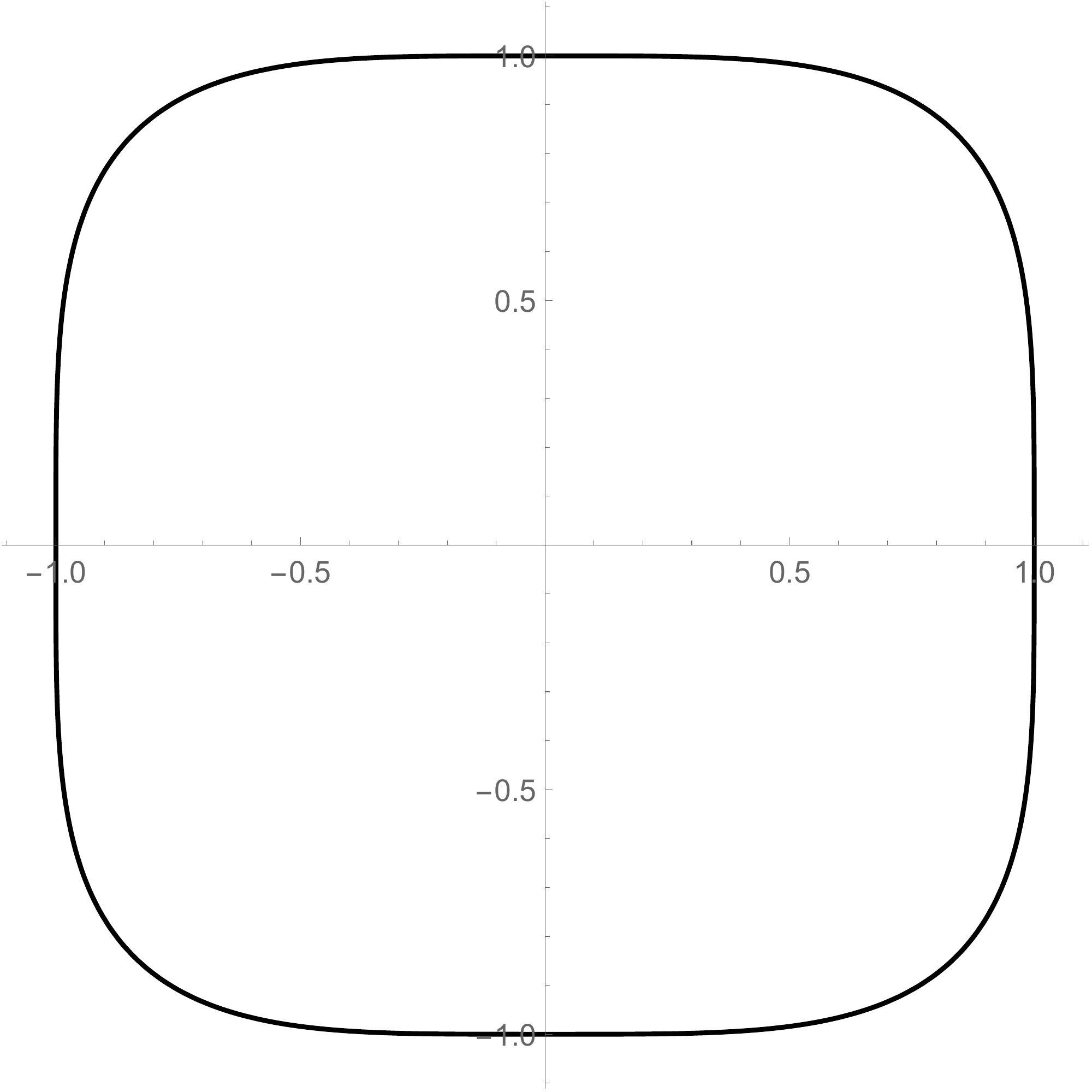}
	\end{subfigure}
	\ \ \ \ \ 
	\begin{subfigure}{0.3\textwidth}
		\centering
		\includegraphics[width=\textwidth]{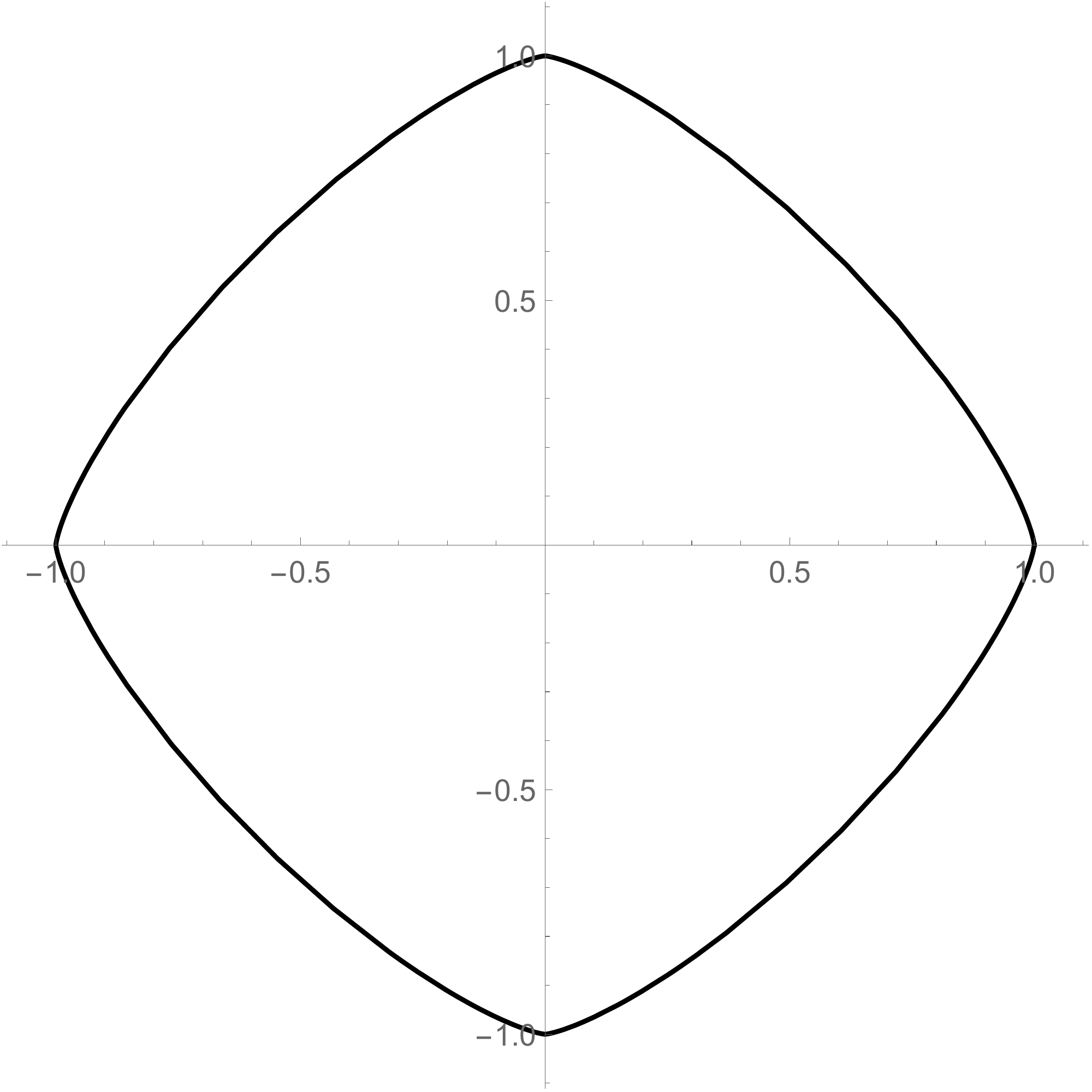}
	\end{subfigure}
	\caption{Sets $\Omega$ (on the left) and $\Omega^\polar$ for $p=4$ and $q=4/3$.}
	\label{fig:lp_lp_balls}
\end{figure}

We will make all computations, assuming that $\varphi\in[0;\pi/2]$, since our results can be easily extended for other intervals $[\pi k/2;\pi(k+1)/2]$, $k\in\Z$.

Let $\theta(\varphi)$ denote the generalized angle corresponding to point $(u(\varphi),v(\varphi))$. Using Theorem~\ref{thm:polar_change}, for $\varphi\in[0;\pi/2]$, we obtain 
\[
	\theta'_{\varphi} = uv'_\varphi - u'_\varphi v = \tfrac1p 4^{\frac1q} \sin^{(\frac1p - \frac1q)}2\varphi.
\]
Function $\theta(\varphi)=\int \theta'_\varphi\,d\varphi$ can be found by substitution $t=\sin^22\varphi$ in terms of incomplete Euler integral of the first kind (beta function) $B(x;a,b)=\int_0^xt^{a-1}(1-t)^{b-1}\,dt$ with $a=\frac12$ and $b=\frac1p$, which can be expressed in terms of hyperheometric function: $B(x;a,b)=\frac1a x^a\,_2F_1(a,1-b;a+1;x)$. Indeed, for $\varphi\in[0;\frac\pi4]$, we have
\[
	\int\sin^{(\frac1p - \frac1q)}2\varphi\,d\varphi = 
	\tfrac14\int(\sin^22\varphi)^{\frac1p-1}(1-\sin^22\varphi)^{-\frac12}\,d(\sin^22\varphi)=
	\tfrac14 B(\sin^22\varphi;\tfrac1p,\tfrac12)+\const.
\]
Using $B(0;a,b)=0$, we obtain
\[
	\theta=\tfrac{1}{p}4^{-\tfrac{1}{p}}B\left(\sin^2 2\varphi;\tfrac1p,\tfrac12\right)\quad\mbox{for}\quad\varphi\in[0;\tfrac\pi4].
\]
Since $\theta$ is the doubled area of the corresponding sector of $\Omega$, the total area $\mathbb{S}$ of $\Omega$ can be express  by $\Gamma$--function:
\begin{equation}
\label{eq:area_lp}
	\mathbb{S} = 4\theta(\pi/4) = \frac{1}{p}4^{\frac{1}{q}}\frac{\Gamma\left(\frac{1}{p}\right)\Gamma(\frac12)}{\Gamma \left(\frac{1}{2}+\frac{1}{p}\right)} =
	\frac{4\Gamma^2\left(1+\frac1p\right)}{\Gamma\left(1+\frac2p\right)}.
\end{equation}
Here we used $B(1;a,b)=\Gamma(a)\Gamma(b)/\Gamma(a+b)$, $\Gamma(\frac12)=\sqrt{\pi}$, and Legendre duplication formula $\Gamma(x)\Gamma(x+\frac12)=2^{1-2x}\sqrt{\pi}\Gamma(2x)$.

Obviously, $\theta(\tfrac\pi2-\varphi)=\tfrac{\mathbb{S}}2-\varphi$ and $\theta(\varphi+\pi k/2)=\theta(\varphi)+\mathbb{S}k/2$ for $k\in\Z$. Hence, for all $\varphi\in\R$, we have
\begin{equation}
\label{eq:theta_lp}
	\theta = \frac{k-[k/2]}{2}\mathbb{S} + (-1)^{k}\tfrac{1}{p}4^{-\tfrac{1}{p}}B\left(\sin^2 2\varphi;\tfrac1p,\tfrac12\right)
	\qquad\mbox{where}\qquad k=\left[\frac{4\varphi}{\pi}\right]=\left[\frac{4\theta}{\mathbb{S}}\right]\in\Z.
\end{equation}
Function in the right hand side is strictly monotone (since its derivative is positive), which allows us to give the following

\begin{definition}
	Put (see Fig.~\ref{fig:cos_sin_lp})
	\[
		\cos_p\theta =|\cos\varphi_p(\theta)|^{2/p}\sgn{\cos\varphi_p(\theta)};\quad
		\sin_p\theta =|\sin\varphi_p(\theta)|^{2/p}\sgn{\sin\varphi_p(\theta)}
	\]
	where $\varphi_p(\theta)$ is a unique solution to~\eqref{eq:theta_lp}.
\end{definition}

\begin{figure}[ht]
	\centering
	\begin{subfigure}{0.45\textwidth}
		\centering
		\includegraphics[width=\textwidth]{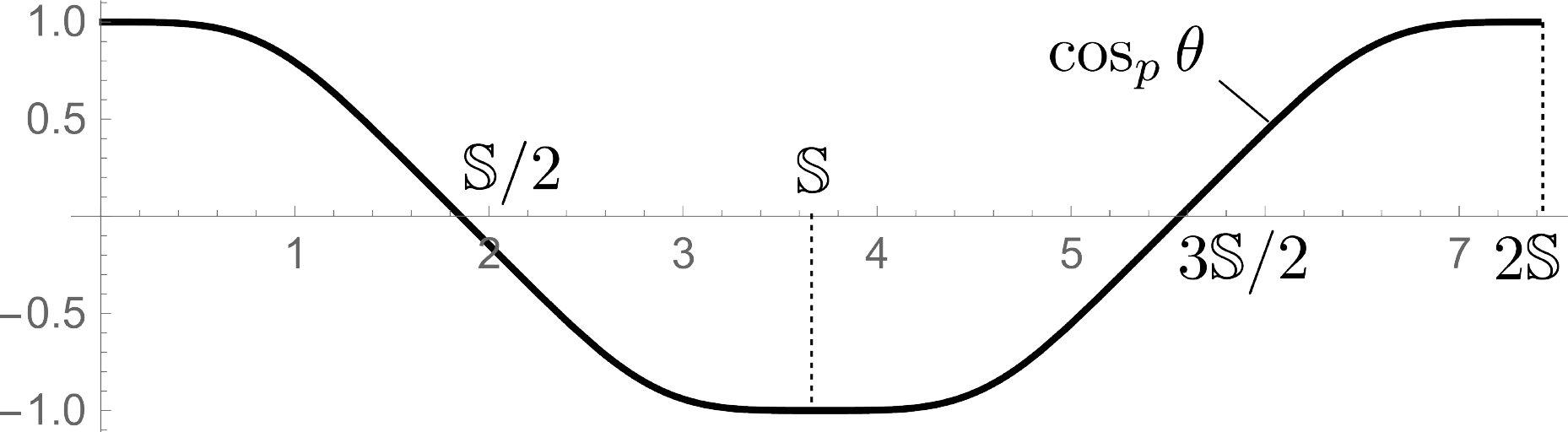}
	\end{subfigure}
	\begin{subfigure}{0.45\textwidth}
		\centering
		\includegraphics[width=\textwidth]{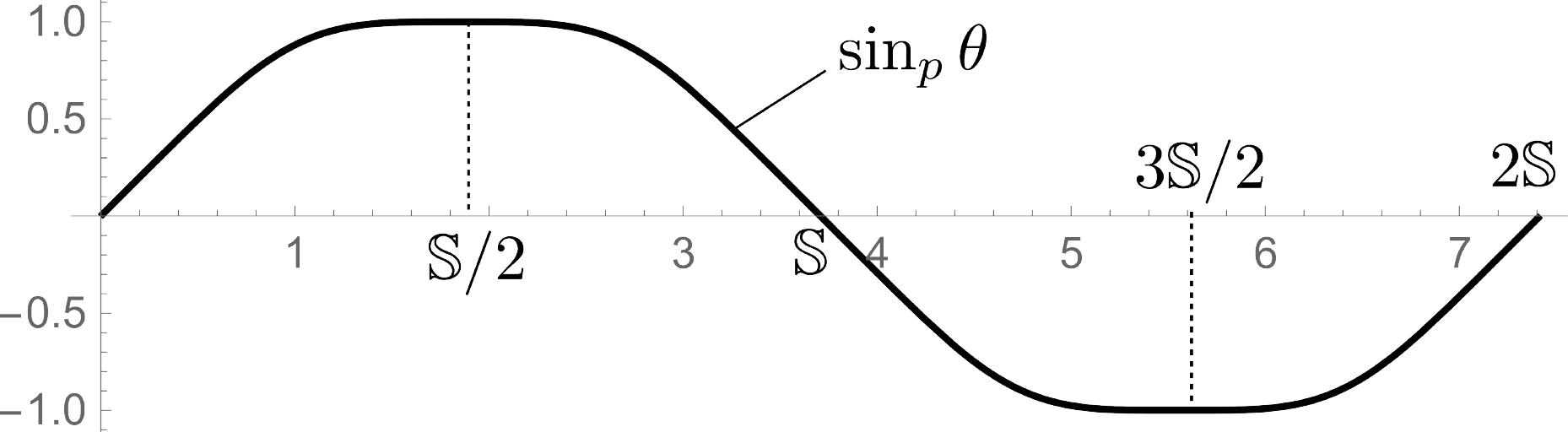}
	\end{subfigure}
	\caption{Graphs of $\cos_\Omega$ and $\sin_\Omega$ for $p=4$.}
	\label{fig:cos_sin_lp}
\end{figure}

Now, we pass to the polar set $\Omega^\polar$, which is the unit ball on $\R^2$ in $L_q$ metric. Let $\theta\xleftrightarrow{\Omega\,}\theta^\polar$. For $\varphi\in[0;\pi/2]$, we obtain
\[
	\cos_{\Omega^\polar}\theta^\polar = \frac{d}{d\theta}\sin_{\Omega}\theta = \frac{1}{\theta'_\varphi}\frac{d}{d\varphi}\sin^{2/p}\varphi = \cos ^{2/q}\varphi;
\]
\[
	\sin_{\Omega^\polar}\theta^\polar = -\frac{d}{d\theta}\cos_{\Omega}\theta = \frac{-1}{\theta'_\varphi}\frac{d}{d\varphi}\cos^{2/p}\varphi = \sin ^{2/q}\varphi.
\]
Hence, convex trigonometry formulae for $\Omega^\polar$ are completely similar to those for $\Omega$ with the same $\varphi$ and can be obtained by substitution $p\leftrightarrow q$ (see Fig.~\ref{fig:cos_sin_lq}):
\begin{equation}
\label{eq:theta_lq}
	\theta^\polar = \frac{k-[k/2]}{2}\mathbb{S^\polar} + (-1)^{k}\tfrac{1}{q}4^{-\tfrac{1}{q}}B\left(\sin^2 2\varphi;\tfrac1q,\tfrac12\right)
\qquad\mbox{where}\qquad k=\left[\frac{4\varphi}{\pi}\right]=\left[\frac{4\theta^\polar}{\mathbb{S^\polar}}\right]\in\Z;
\end{equation}
\begin{equation}
\label{eq:area_lq}
	\mathbb{S}^\polar=\frac{4\Gamma^2\left(1+\frac1q\right)}{\Gamma\left(1+\frac2q\right)}.
\end{equation}

So, we have proved the following

\begin{theorem}
\label{thm:cos_sin_lp}
	Let $\Omega=\{(u,v):|u|^p+|v|^p\le1\}\subset \R^2$ with $1<p<\infty$, and $pq=p+q$. Then
	\[
		\cos_\Omega\theta=\cos_p\theta\qquad\text{and}\qquad
		\sin_\Omega\theta=\sin_p\theta.
	\]
	The angle $\theta^\polar$ corresponding to $\theta$ is unique (up to the period $2\mathbb{S}^\polar$) and is given as a unique solution to the equation $\varphi_q(\theta^\polar)=\varphi_p(\theta)$ (which is equivalent to pair~\eqref{eq:theta_lp} and~\eqref{eq:theta_lq}), and
	\[
		\cos_{\Omega^\polar}\theta^\polar=\cos_q\theta^\polar\qquad\text{and}\qquad
		\sin_{\Omega^\polar}\theta^\polar=\sin_q\theta^\polar.
	\]
	
\end{theorem}

\begin{figure}[ht]
	\centering
	\begin{subfigure}{0.45\textwidth}
		\centering
		\includegraphics[width=\textwidth]{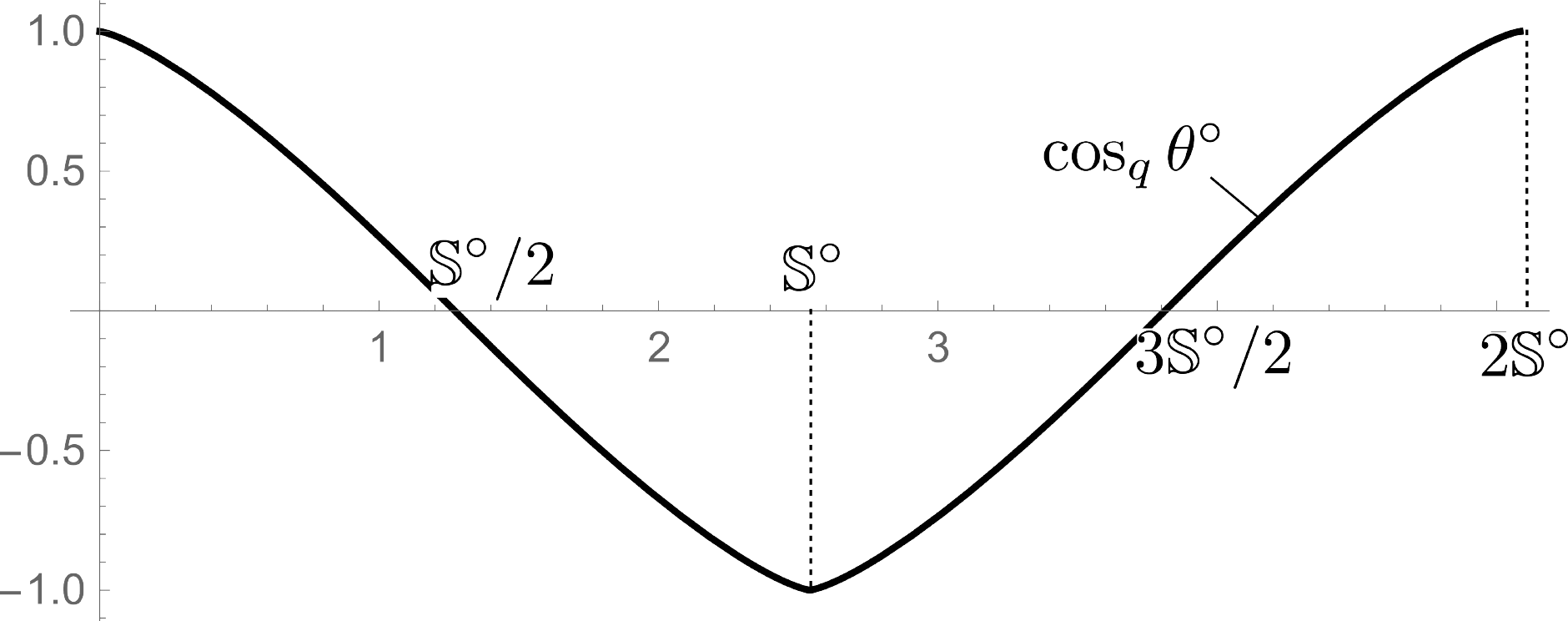}
	\end{subfigure}
	\begin{subfigure}{0.45\textwidth}
		\centering
		\includegraphics[width=\textwidth]{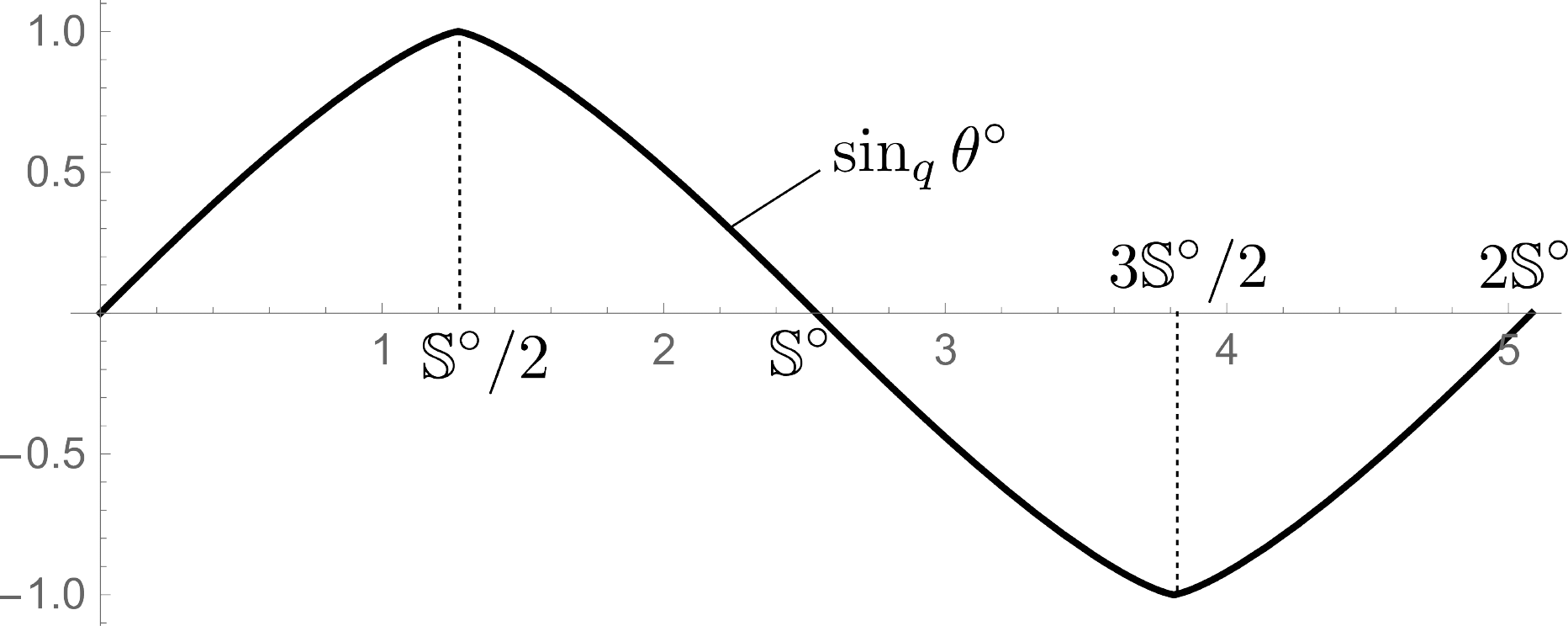}
	\end{subfigure}
	\caption{Graphs of $\cos_{\Omega^\polar}$ and $\sin_{\Omega^\polar}$ for $q=4/3$ (i.e. $p=4$).}
	\label{fig:cos_sin_lq}
\end{figure}

\begin{figure}[ht]
	\centering
	\includegraphics[width=0.4\textwidth]{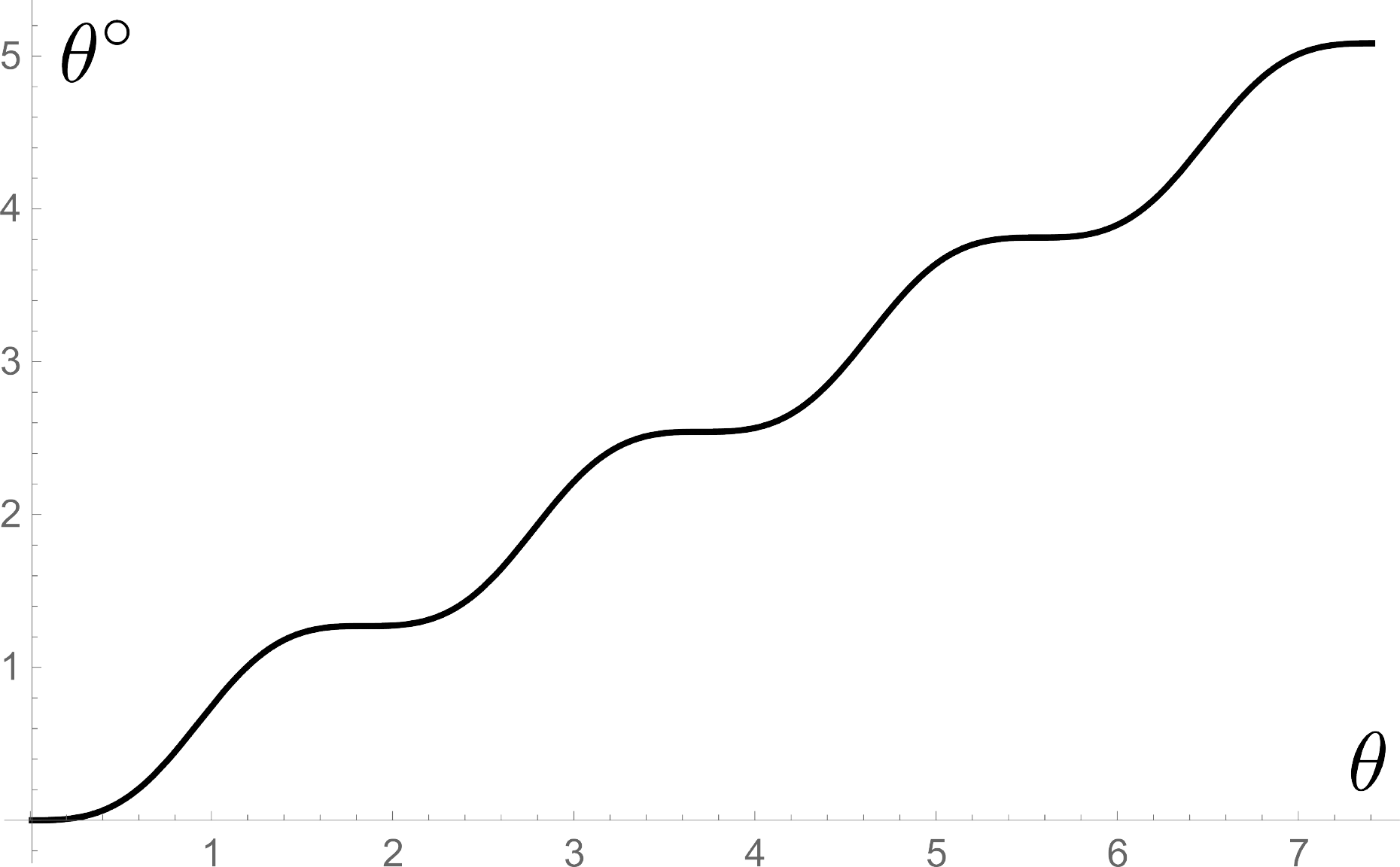}
	\caption{Graph of $\theta^\polar(\theta)$ for $p=4$.}
	\label{fig:theta_theta_polar}
\end{figure}

\begin{remark}
	Using Theorems~\ref{thm:cos_sin_lp} and~\ref{thm:derivatives} we obtain 
	\[
		\frac{d}{d\theta} \cos_p\theta = -\sin_q\theta^\polar = -|\sin\varphi|^{2/q}\sgn\sin\varphi = -|\sin_p\theta|^{p-1}\sgn\sin_p\theta
	\]
	\[
		\frac{d}{d\theta} \sin_p\theta = \cos_q\theta^\polar = |\cos\varphi|^{2/q}\sgn\sin\varphi = |\cos_p\theta|^{p-1}\sgn\cos_p\theta
	\]
	Therefore, $\cos_p$ and $\sin_p$ are solutions to the following ODE
	\begin{equation}
		\frac{d}{d\theta}x = -|y|^{p-1}\sgn y,\quad \frac{d}{d\theta}y = |x|^{p-1}\sgn x,\quad x(0)=1,\quad y(0)=0.
	\label{eq:Shelupskiy}
	\end{equation}
	Solutions to this system are known as Shelupsky’s generalized trigonometric functions \cite{Shelupsky}. Shelupsky's functions form 1-parametric family (determined by the parameter $p$) of pair of functions $(x(\theta),y(\theta))$ satisfying~\eqref{eq:Shelupskiy}. The convex triginometry functions introduced in~\cite{CT1} form infinite-dimensional family, which is determined by the set $\Omega$. So, we have proved that Shelupsky's functions are the special case of convex trigonometry functions when $\Omega$ is the unit $L_p$-ball. Hence, Shelupsky's functions share all properties of convex trigonometry. In particular, the correspondence $\theta\xleftrightarrow{\Omega\,}\theta^\polar$ given by~\eqref{eq:theta_lp} and~\eqref{eq:theta_lq} shows very a convenient connection between pairs ($\cos_p$, $\sin_p$) and ($\cos_q$,$\sin_q$) of Shelupsky's functions for $pq=p+q$ by Theorem~\ref{thm:derivatives}.

	Shelupsky's functions are widely used for studying $p$-Laplacian eigenvalues and eigenfunctions. A nice introduction and intergal representation for the inverse functions can be found in \cite{WeiLiuElgindi}. 
\end{remark}

Now, since we have obtained formulae for $\cos_\Omega$, $\sin_\Omega$, $\cos_{\Omega^\polar}$, and $\sin_{\Omega^\polar}$, we are ready to formulate the result on extremals in $L_p$ case.

\begin{proposition}
\label{prop:lp_case}
	Suppose that $U=\{(\dot x,\dot y)\in\R^n\times\R^{n*}:\sum_{i=1}^n(|\dot x_i|^p+|\dot y_i|^p)\le 1\}$ for $1<p<\infty$. Then for any extremal $(x(t),y(t),z(t))$ in problem~\eqref{eq:main_problem} on $\Hb{2n+1}$, there exists constants $\gamma=0,\pm1$ and $A=(A_1,\ldots,A_n)\in\R^n_+$ not vanishing at the same time such that for $\alpha=(\sum_j A_j^q)^{1/q}$ (where $pq=p+q$), we have
	\begin{enumerate}
		\item If $\gamma\ne 0$, then for each $i$ with $A_i>0$, there exists a constant $\theta_{i0}^\polar\in\R$ such that
		\[
			x_i=\gamma A_i\left(
				\sin_q\theta_i^\polar - \sin_q\theta_{i0}^\polar
			\right)
			\qquad\text{and}\qquad
			y_i=\gamma A_i\left(
				\cos_q\theta_{i0}^\polar - \cos_q\theta_i^\polar
			\right)
		\]
		where $\theta_i^\polar=\theta_{i0}^\polar + \gamma A_i^{q-2}\alpha^{-q/p}t$; for each $i$ with $A_i=0$, we have $x_i\equiv y_i\equiv 0$; and
		\[
			2z=\gamma\alpha t +
					\sum_{i=1}^n A_i^2\Big(
						\sin_q\theta_{i0}^\polar\cos_q\theta_i^\polar - \cos_q\theta_{i0}^\polar\sin_q\theta_i^\polar
					\Big).
		\]
		\item If $\gamma=0$, then for each $i$ with $A_i>0$, we have
		\[
			x_i = A_i^{q-1}\alpha^{q/p} u_i t
			\quad\mbox{and}\quad
			y_i = A_i^{q-1}\alpha^{q/p} v_i t		
		\]
		where $(u_i,v_i)\in\partial\Omega$ is a fixed point; for each $i$ with $A_i=0$, we have $x_i\equiv y_i\equiv 0$; and $z\equiv 0$.
	\end{enumerate}
	Moreover, if a trajectory $(x(t),y(t),z(t))$ has one of the described forms, then it is an extremal in problem~\eqref{eq:main_problem} on $\Hb{2n+1}$.
\end{proposition}

\begin{proof}

First, let us compute $\lambda=(\lambda_1,\ldots,\lambda_n)\in\partial s_\Xi(A)$. Since $\Xi=\{\lambda:\lambda_i\ge 0\mbox{ and }\sum_i\lambda_i^p\le 1\}$, we have $\Xi^\polar=\{A:\sum_i\max\{0,A_i\}^q\le 1\}$. Hence 
\[
	s_{\Xi}(A) = \mu_{\Xi^\polar} (A) =  \left(\sum_i\max\{0,A_i\}^q\right)^{1/q}.
\]
Therefore, if $A\in\R^n_+$, then $s_\Xi(A)=\alpha$, $\partial s_{\Xi}(A) = \{s'_\Xi(A)\}$. Hence, if $A_i>0$, then $\lambda_i=\partial s_\Xi(A)/\partial A_i = \alpha^{-1/p}A_i^{q-1}$, and if $A_i=0$ then $\lambda_i=0$ (see Corollary~\ref{cor:mu_strictly_monotone}). So, in both cases, $\lambda_i$ is constant (see Corollary~\ref{cor:stricly_convex}). 

If $\gamma\ne 0$, then $\theta_i^\polar = \theta_{i0}^\polar + \gamma\alpha^{-1/p}A_i^{q-2}t$ and item 1 follows from Theorem~\ref{thm:main}. If $\gamma=0$ and $A_i>0$ for some $i$, then $u_i=\cos_\Omega\theta_i$ and $v_i=\sin_\Omega\theta_i$ are constants, since $\Omega$ is strictly convex (see Corollary~\ref{cor:stricly_convex}). If $\gamma=0$ and $A_i=0$ for some $i$, then $x_i\equiv y_i\equiv0$, since $\lambda_i\equiv0$. Hence, if $\gamma=0$, then $\dot z\equiv 0$.
\end{proof}

\section{Convex hull and direct product of compact convex sets}
\label{sec:convex_hull_and_direct_product}

In this section, we consider the following two cases. First, if $U$ is a convex hull of sets $\Omega_i$ lying on the planes $O\dot x_i\dot y_i$, then it satisfies the main assumption with $\mu(\lambda)=\sum_{i=1}^n\lambda_i$. This case includes $L_1$--case $U=\{(\dot x,\dot y):\sum_{i=1}^n(|\dot x_i|+|\dot y_i|)\le 1\}$, which appears if all sets $\Omega_i$ coincide with the set $\Omega=\{(u,v)\subset\R^2:|u|+|v|\le 1\}$. For this particular set $\Omega$ functions $\cos_\Omega$ and $\sin_\Omega$ were computed in~\cite[Example 4]{CT1}. They are both periodic functions with period $4$, and (see Fig.~\ref{fig:cos_sin_l1})
\[
	\cos_\Omega\theta = |\theta-2| - 1\ \text{if}\ \theta\in[0;4]
	\quad\mbox{and}\quad
	\sin_\Omega\theta = \cos_\Omega(\theta-1)
\]

\begin{figure}[ht]
	\centering
	\begin{subfigure}{0.3\textwidth}
		\centering
		\includegraphics[width=\textwidth]{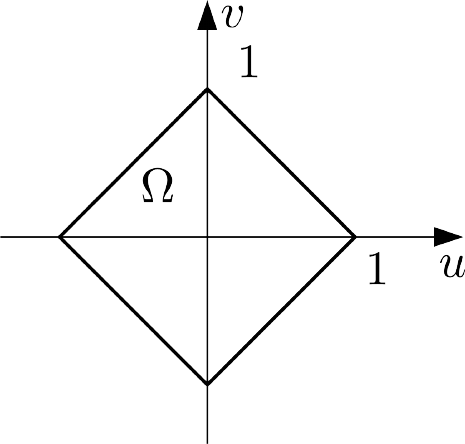}
	\end{subfigure}
	\ \ \ 
	\begin{subfigure}{0.3\textwidth}
		\centering
		\includegraphics[width=\textwidth]{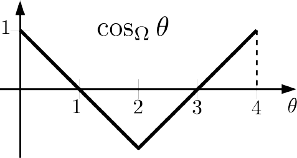}
	\end{subfigure}
	\ \ \ 
	\begin{subfigure}{0.3\textwidth}
		\centering
		\includegraphics[width=\textwidth]{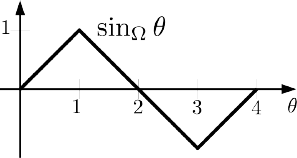}
	\end{subfigure}
	\caption{Graphs of $\cos_\Omega$ and $\sin_\Omega$ for $\Omega=\{u,v:|u|+|v|\le 1\}$.}
	\label{fig:cos_sin_l1}
\end{figure}

Second, if $U$ is a direct product of sets $\Omega_i$ lying on the planes $O\dot x_i\dot y_i$, then it satisfies the main assumption with $\mu(\lambda)=\max_i\lambda_i$. This case includes $L_\infty$--case $U=\{(\dot x,\dot y):\max_i\{|\dot x_i|,|\dot y_i|\}\le 1\}$, which appears if all sets $\Omega_i$ coincide with the set $\Omega^\polar=\{(u,v)\subset\R^2:\max\{|u|,|v|\}\le 1\}$. For this particular set $\Omega^\polar$ functions $\cos_{\Omega^\polar}$ and $\sin_{\Omega^\polar}$ were computed in~\cite[Example 4]{CT1}. They both have period 8, and  (see Fig.~\ref{fig:cos_sin_linfty})
\[
	\cos_{\Omega^\polar}\theta^\polar = \begin{cases}
		\frac12|\theta^\polar-3| + \frac12|\theta^\polar-5|-2,&\mbox{if }\theta^\polar\in[1;7];\\
		1,&\mbox{if }\theta^\polar\in[0;1]\cup[7;8];
	\end{cases}
	\quad\mbox{and}\quad
	\sin_{\Omega^\polar}\theta^\polar = \cos_{\Omega^\polar}(\theta^\polar-2).
\]

\begin{figure}[ht]
	\centering
	\begin{subfigure}{0.3\textwidth}
		\centering
		\includegraphics[width=\textwidth]{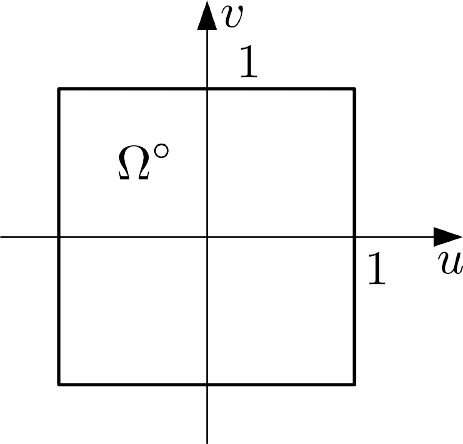}
	\end{subfigure}
	\ \ \ 
	\begin{subfigure}{0.3\textwidth}
		\centering
		\includegraphics[width=\textwidth]{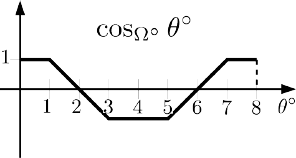}
	\end{subfigure}
	\ \ \ 
	\begin{subfigure}{0.3\textwidth}
		\centering
		\includegraphics[width=\textwidth]{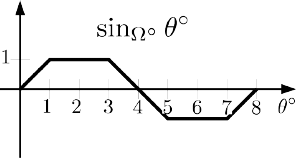}
	\end{subfigure}
	\caption{Graphs of $\cos_{\Omega^\polar}$ and $\sin_{\Omega^\polar}$ for $\Omega^\polar=\{u,v:\max\{|u|,|v|\}\le 1\}$.}
	\label{fig:cos_sin_linfty}
\end{figure}

\begin{proposition}
\label{prop:conv_case}
	Suppose that $U$ is a convex hull of compact convex sets $\Omega_i$ lying on the planes $O\dot x_i\dot y_i$ containing the origin in their (2--dimensional relative) interiors. Then in Theorem~\ref{thm:main}, we have $s_\Xi(A)=\max_i A_i$ and $\lambda(t)\in\R^n_+$ is an arbitrary measurable function such that (1) $\sum_i\lambda_i(t)=1$; and (2) if $A_i<s_\Xi(A)$ for some $i$, then $\lambda_i \equiv0$. Particularly, $x_i\equiv y_i\equiv 0$ if $A_i=0$.
\end{proposition}

\begin{proof}
	Since $\mu(\lambda)=\sum_i\lambda_i$, we have $\Xi=\{\lambda:\sum_i\lambda_i\le 1,\ \forall i\,\lambda_i\ge 0\}$. Therefore, $\Xi^\polar=\{A:\max_iA_i\le 1\}$ and $s_\Xi(A)=\mu_{\Xi^\polar}(A)=\max_iA_i$. Hence, using Dubovitskiy-Milyutin formula for the subdifferential of maximum (see~\cite[Section~1.5]{TikhomirovMagaril}), we have
	\[
		\partial s_\Xi(A) = \left\{
			\lambda:\sum_i\lambda_i=1,\ \forall i\,\lambda_i\ge0\mbox{ and }\big(A_i<s_\Xi(A)\Rightarrow\lambda_i=0\big)
		\right\}.
	\]
\end{proof}

\begin{proposition}
\label{prop:prod_case}
	Suppose that $U$ is a direct product of compact convex sets $\Omega_i$ lying on the planes $O\dot x_i\dot y_i$ containing the origin in their (2--dimensional relative) interiors. Then in Theorem~\ref{thm:main}, we have (1) $s_\Xi(A)=\sum_i A_i$ for $A\in\R_+^n$ and (2) $\lambda(t)\in\R^n_+$ is an arbitrary measurable function such that $\lambda_i\equiv 1$ for each index $i$ with $A_i>0$. In particular, if $\gamma\ne 0$, then $\theta_i^0(t)$ is linear in $t$ for each $i$ with $A_i>0$.
\end{proposition}

\begin{proof}
	Since $\mu(\lambda)=\max_i\lambda_i$, we have $\Xi=\{\lambda:\forall i\,0\le\lambda_i\le 1\}$. Therefore, $\Xi^\polar=\{A:\sum_{i:A_i\ge0}A_i\le 1\}$ and $s_\Xi(A)=\mu_{\Xi^\polar}(A)=\sum_i\max\{0,A_i\}$. Hence, for $A\in\R^n_+$, we have
	\[
		\partial s_\Xi(A) = \argmax_{\lambda\in\Xi} \sum_iA_i\lambda_i=
		\left\{
			\lambda\in\Xi:\forall i\, (A_i>0\Rightarrow \lambda_i=1)
		\right\}.
	\]
\end{proof}

\begin{remark}
\label{rm:arbitrary_symplectic_planes}
	Moreover, these results can be applied to convex hulls and sums of $n$ convex compact 2--dim sets lying in arbitrary planes $P_i$ if $\R^n\times\R^{n*}=\sum_{i=1}^nP_i$ and these planes are pairwise skew-orthogonal w.r.t.\ canonical symplectic structure on $\R^n\times \R^{n*}$.
\end{remark}

\section{Arbitrary sub--Riemannian case}
\label{sec:subriemannain}

In this section, we consider the case when $U$ is an arbitrary ellipsoid $U=\{(\dot x,\dot y):g(\dot x,\dot y)\le 1\}\subset\R^n\times\R^{n*}$ where $g$ is a symmetric positive definite bilinear form on $\R^n\times\R^{n*}$. We claim, that in this case, the explicit formulae for extremals can be obtained by Theorem~\ref{thm:main}. The following theorem was obtained in~\cite{BiggsNagy} using authomormisms of $\Hb{2n+1}$. We prefer to give another proof by almost complex structures.

\begin{theorem}[{\cite[Theorem 3]{BiggsNagy}}]
\label{thm:symplectic_change}
	There exists a linear symplectic change of variable $(x,y)=C(\tilde x,\tilde y)$ (i.e.\ with $C\in \mathrm{Sp}(\R^n\times\R^{n*})$) such that control system~\eqref{eq:main_problem} takes the form
	\[
		\sum_i\tfrac{1}{a_i^2}(\dot{\tilde x}_i^2 + \dot{\tilde y}_i^2)\le 1;\qquad
		\dot z= \frac12\sum_i(\tilde x_i\dot{\tilde y}_i-\tilde y_i\dot{\tilde x}_i).
	\]
	where $a_i>0$ are some constants.
\end{theorem}

\begin{proof}
Denote by $\omega$ the canonical skew-symmetric product (symplectic form) on $\R^n\times\R^{n*}$:
\[
	\omega((x^1,y^1),(x^2,y^2)) = \langle x^1,y^2\rangle - \langle x^2, y^1\rangle.
\]
Consider the following symmetric, positive definite form $p=\omega^Tg^{-1}\omega = -\omega g^{-1}\omega$. There exists an orthogonal (w.r.t.~$g$) decomposition $\R^n\times\R^{n*}=\bigoplus_{k=1}^m V_j$, $V_k\perp V_l$ for $k\ne l$, such that $g^{-1}pV_k=V_k$ and
\[
	\left(g^{-1}p\right)\big|_{V_k}=c_k^2\id_{V_k}\quad\text{for some}\quad c_k>0
\]
where $c_k^2$ are distinguish eigenvalue of $g^{-1}p$. Since $g^{-1}p$ commutes with $g^{-1}\omega$ and $c_k\ne c_l$ for $k\ne l$, we also have $g^{-1}\omega V_k=V_k$.

Consider an operator $J\in\mathrm{GL}(\R^n\times\R^{n*})$ such that $J|_{V_k}=c_k^{-1}(g^{-1}\omega)$. We claim that $J$ is an almost complex structure on $\R^n\times\R^{n*}$ compatible with $g$ (see~\cite[Section 1.2]{Huybrechts}). Indeed, $J$ is orthogonal (w.r.t.~$g$), since $J(V_k)=V_k$, and
\[
	(J^TgJ)|_{V_k} = c_k^{-1} (\omega^T(g^T)^{-1}gg^{-1}\omega)|_{V_k} = -c_k^{-2} (\omega g^{-1}\omega)|_{V_k} =
	c_k^{-2}(gg^{-1}p)|_{V_k} = g|_{V_k},
\]
\noindent and
\[
	J^2|_{V_k} = c_k^{-2}(g^{-1} \omega g^{-1}\omega)|_{V_k} = -c_k^{-2}(g^{-1}p)|_{V_k} = -\id_{V_k}.
\]
Hence, $H(a,b)=g(a,b)+\mathbf{i}\,g(Ja,b)\in\mathbb{C}$ for $a,b\in \R^n\times\R^{n*}$, is a hermitian inner product on~$\R^n\times\R^{n*}$ compatible with almost complex structure given by $J$. Subspaces $V_k$ are almost complex subspaces, since $JV_k=V_k$, and they are  orthogonal w.r.t.\ $H$, since $g[V_k,V_l]=0$ for $k\ne l$. Therefore, there exists an orthonormal w.r.t.~$g|_{V_k}$ basis on~$V_k$ that is also symplectic w.r.t.\ $gJ$. Remind that $g^{-1}\omega V_k=V_k$ and $g^{-1}\omega|_{V_k}=c_kJ$. Hence, collecting these bases for all $V_k$, $k=1,\ldots,m$, we obtain an orthonormal w.r.t.~$g$ basis $e_1,f_1,e_2,f_2,\ldots,e_n,f_n$ such that matrix of $\omega$ in this basis is block diagonal with $2\times2$ blocks on the diagonal of the form 
\[
	\begin{pmatrix}
		0 & -a_i^2\\
		a_i^2& 0
	\end{pmatrix}
\]
where $a_i>0$, $i=1,\ldots,n$. Let $C$ denote the change of variable on $\R^n\times\R^{n*}$ from the initial standard basis to the basis $a_1^{-1}e_1,\ldots,a_n^{-1}e_n,a_1^{-1}f_1,\ldots a_n^{-1}f_n$. Note that $C\in\mathrm{Sp}(\R^n\times \R^{n*})$, since it preserves $\omega$. In other words, if $(x,y)=C(\tilde x,\tilde y)$, then $\omega(\tilde x,\tilde y)=\omega(x,y)$. Moreover, $C^TgC=\mathrm{diag}(a_1^{-2},\ldots,a_n^{-2},a_1^{-2},\ldots,a_n^{-2})$. Therefore,
\[
	U=\left\{(\tilde x,\tilde y):\sum_i\frac{1}{a_i^2}\left(\dot{\tilde x}_i^2 + \dot{\tilde y}_i^2\right)\le 1\right\}.
\]

Let us now compute $\dot z$ in new coordinates. Obviously,
\[
	\dot z = \frac12\sum_i(x_i\dot y_i - \dot x_i y_i)= \frac12\omega((x,y),(\dot x,\dot y)) =
	\omega((\tilde x,\tilde y),(\dot{\tilde x},\dot{\tilde y}))=
	\frac12\sum_i(\tilde x_i\dot{\tilde y}_i - \dot{\tilde x}_i\tilde y_i).
\]
\end{proof}

So, any sub-Riemannain problem~\eqref{eq:main_problem} on $\Hb{2n+1}$ is equivalent to the following simplest one (by an appropriate symplectic change of variables given in the proof of Theorem~\ref{thm:symplectic_change}):
\begin{equation}
\label{eq:elliptic_control_system}
	\begin{array}{c}
		T\to\inf;\\
		\sum_i\frac{1}{a_i^2}(\dot x_i^2 + \dot y_i^2)\le 1;\qquad \dot z = \frac12\sum_i(x_i\dot y_i-\dot x_i y_i);\\
		\forall i\quad x_i(0)=y_i(0)=z(0)=0.
	\end{array}
\end{equation}

This problem satisfies our main assumption with $\mu(\lambda)=(\sum_ia_i^{-2}\lambda_i^2)^{1/2}$ and all $\Omega_i$ being unit discs.

\begin{proposition}
	For any extremal $(x(t),y(t),z(t))$ in problem~\eqref{eq:elliptic_control_system} on $\Hb{2n+1}$, there exists constants $\gamma=0,\pm1$ and $A=(A_1,\ldots,A_n)\in\R^n_+$ not vanishing at the same time and constants $\theta_{i0}\in\R$ for all $i$ with $A_i>0$ such that for $\alpha=(\sum_i a_i^2 A_i^2)^{1/2}$, we have
		\begin{enumerate}
			\item If $\gamma\ne 0$, then for each $i$ with $A_i>0$, we have
			\[
				x_i=\gamma A_i\left(
					\sin\theta_i - \sin\theta_{i0}
				\right)
				\quad\mbox{and}\quad
				y_i=\gamma A_i\left(
					\cos\theta_{i0} - 
					\cos\theta_i
				\right)
			\]
			where $\theta_i=\theta_{i0} + \gamma a_i^2t/\alpha$; for each $i$ with $A_i=0$, we have $x_i\equiv y_i\equiv 0$; and
			\[
				2z=\gamma\alpha t -	\sum_{i:A_i>0} A_i^2\sin\left(\gamma a_i^2t/\alpha\right).
			\]
			\item If $\gamma=0$, then for each $i$ with $A_i>0$, we have
			\[
				x_i = \frac{a_i^2A_i}{\alpha}t\cos\theta_{i0}
				\quad\mbox{and}\quad
				y_i = \frac{a_i^2A_i}{\alpha}t\sin\theta_{i0};
			\]
			for each $i$ with $A_i=0$, we have $x_i\equiv y_i\equiv 0$; and $z\equiv0$.
		\end{enumerate}
		Moreover, if a trajectory $(x(t),y(t),z(t))$ has one the described forms, then it is an extremal in problem~\eqref{eq:elliptic_control_system} on $\Hb{2n+1}$.
\end{proposition}

\begin{proof}

	Proof is completely trivial since $\Omega_i=\Omega_i^\polar$ are unit euclidean discs. Hence, correspondence $\theta^\polar\xleftrightarrow{\Omega_i}\theta^\polar$ is equivalent to equality $\theta=\theta^\polar$, and $\sin_{\Omega_i}=\sin$ and $\cos_{\Omega_i}=\cos$. Therefore, applying Theorem~\ref{thm:main} needs only a computation of $s_\Xi(A)$ for $A\in\R^n_+$. Since $\Xi=\{\lambda\in\R^n_+:\sum_i a_i^{-2}\lambda_i^2\le 1\}$, then $\Xi^\polar=\{A:\sum_i a_i^2\max\{0,A_i\}^2\le x\}$, $\mu_{\Xi^\polar}=s_{\Xi}=(\sum_i a_i^2\max\{0,A_i\}^2)^{1/2}$, and $s_\Xi(A)=\sum_i a_i^2A_i^2$ for $A\in\R^n_+$. Hence $\lambda_i=\partial s_\Xi/\partial A_i = a_i^2A_i/\alpha$.

\end{proof}

The author would like to express his deep gratitude to Professor A.I.~Nazarov for interesting discussions and pointing on a relation to Shelupsky's functions.

\printbibliography

\end{document}